%% file: M101628.tex
\documentclass[final,leqno,onefignum,onetabnum]{siamltex1213}
\usepackage{tabu}
\input{cm_defs_simax}
\newcommand\citep\cite
\newcommand\citet\cite

\numberwithin{equation}{section}

\title{\textbf{Scalable Robust Matrix Recovery:} \\
{Frank-Wolfe Meets Proximal Methods}}
\author{Cun Mu\footnotemark[1]
  \and Yuqian Zhang\footnotemark[2]
  \and John Wright\footnotemark[2]
  \and Donald Goldfarb\footnotemark[1]}

\begin{document}
\maketitle
\slugger{sisc}{xxxx}{xx}{x}{x--x}%slugger should be set to mms, siap, sicomp, sicon, sidma, sima, simax, sinum, siopt, sisc, or sirev

{\renewcommand{\thefootnote}{\fnsymbol{footnote}}%
\footnotetext[1]{Department of Industrial Engineering and Operations Research, Columbia University (\url{cm3052@columbia.edu}, \url{goldfarb@columbia.edu}). DG was funded by NSF Grants DMS-1016571 and CCF-1527809.}
\footnotetext[2]{Department of Electrical Engineering, Columbia University, (\url{yq2409@cs.columbia.edu}, \url{johnwright@ee.columbia.edu}). JW was funded by ONR-N00014-13-0492.}

\begin{abstract}
Recovering matrices from compressive and grossly corrupted observations is a fundamental problem in robust statistics, with rich applications in computer vision and machine learning. In theory, under certain conditions, this problem can be solved in polynomial time via a natural convex relaxation, known as Compressive Principal Component Pursuit (CPCP). However, many existing provably convergent algorithms for CPCP suffer from superlinear per-iteration cost, which severely limits their applicability to large-scale problems. In this paper, we propose provably convergent, scalable and efficient methods to solve CPCP with (essentially) linear per-iteration cost. Our method combines classical ideas from Frank-Wolfe and proximal methods. In each iteration, we mainly exploit Frank-Wolfe to update the low-rank component with rank-one SVD and exploit the proximal step for the sparse term. Convergence results and implementation details are discussed. We demonstrate the practicability and scalability of our approach with numerical experiments on visual data.
\end{abstract}



\pagestyle{myheadings}
\thispagestyle{plain}
\markboth{C. MU, Y. ZHANG, J. WRIGHT AND D. GOLDFARB}{SCALABLE ROBUST MATRIX RECOVERY}

\section{Introduction}\label{sec:intro}
Suppose that a matrix $\mb M_0 \in \reals^{m\times n}$ is of the form $\mb M_0 = \mb L_0 + \mb S_0 + \mb N_0$, where $\mb L_0$ is a low-rank matrix, $\mb S_0$ is a sparse error matrix, and $\mb N_0$ is a dense noise matrix. Linear measurements
\begin{equation}
\bm b \;=\; \mc A[\mb M_0] \;= \; \big(\innerprod{\mb A_1}{\mb M_0}, \innerprod{\mb A_2}{\mb M_0}, \ldots, \innerprod{\mb A_p}{\mb M_0}   \big)^{\t} \in \reals^p
\end{equation}
are collected, where $\mc A: \; \reals^{m\times n} \to \reals^p$ is the sensing operator, $\mb A_k$ is the sensing matrix for the $k$-th measurement and $\innerprod{\mb A_k}{\mb M_0}  \doteq \mbox{Tr}(\mb M_0^{\t} \mb A_k)$.
%and $\innerprod{\mb A_k}{\mb M_0} = \sum_i\sum_j (\mb A_k)_{ij} (\mb M_0)_{ij}.$
{\em Can we, in a tractable way,  recover $\mb L_0$ and $\mb S_0$ from $\bm b$, given $\mc A$?}

One natural approach is to solve the optimization combining the fidelity term and the structural terms:
\begin{equation}
\min_{\mb L, \mb S} \;\; \frac{1}{2} \norm{\bm b - \mc A [\mb L+\mb S]}{2}^2 + \lambda_L \mbox{rank}(\mb L) + \lambda_S \norm{\mb S}{0}.
\label{eqn:nonconvex_prob}
\end{equation}
Here, $\lambda_L$ and $\lambda_S$ are regularization parameters, and $\norm{\mb S}{0}$ denotes the number of nonzero entries in $\mb S$.
%$\PO[\cdot]$ denotes the
%orthogonal projection onto the linear space of matrices supported on
%$\Omega$, i.e.,
%$
%\PO [\mb M_0] (i,j) = M_{ij}$ if $(i,j) \in \Omega$ and $\PO[\mb M_0](i,j) = 0$ %otherwise, $\lambda_L$ and $\lambda_S$ are regularization parameters, and $\norm{\mb S}{0}$ denotes the number of nonzero entries in $\mb S$.

Unfortunately, problem \eqref{eqn:nonconvex_prob} is nonconvex, and hence is not directly tractable. However, by replacing the $\ell_0$ norm $\norm{\mb S}{0}$ with the $\ell_1$ norm $\norm{\mb S}{1} \doteq \sum_{i=1}^m \sum_{j=1}^n |S_{ij}|$, and replacing the rank  $\mbox{rank}(\mb L)$ with the nuclear norm $\norm{\mb L}{*}$ (defined as the sum of the singular values of $\mb L$), we obtain a natural, tractable, convex relaxation of \eqref{eqn:nonconvex_prob},
\begin{equation}
\min_{\mb L, \mb S} \;\; \frac{1}{2} \norm{\bm b - \mc A [\mb L+\mb S]}{2}^2 + \lambda_L \norm{\mb L}{*} + \lambda_S \norm{\mb S}{1}.
\label{eqn:penalized_prob_first}
\end{equation}
This convex surrogate is sometimes referred to as {\em compressive principal component pursuit (CPCP)} \cite{wright2013compressive}. Equivalently, since $$\set{\;\mb M\in \reals^{m\times n} \;\;\vert\;\; \bm b = \mc A[\mb M]\;}\;\; = \;\; \set{\;\mb M\in \reals^{m\times n} \;\;\vert\;\; \PQ[\mb M] = \PQ[\mb M_0]\;},$$
where $\mc Q \subseteq \reals^{m \times n}$ is a linear subspace spanned by the set of sensing matrices $\set{\mb A_i}_{i=1}^p$, and $\PQ$ denotes the projection operator onto that subspace, we can rewrite problem \eqref{eqn:penalized_prob_first} in the (possibly) more compact form, \footnote{
To transform problem \eqref{eqn:penalized_prob_first} into problem \eqref{eqn:penalized_prob}, simple procedures like Gram-朣chmidt might be invoked.  Despite being equivalent, one formulation might be preferred over the other in practice, depending on the specifications of the sensing operator $\mc A[\cdot]$. In this paper, we will mainly focus on solving problem \eqref{eqn:penalized_prob} and its variants. Our methods, however, are not restrictive to \eqref{eqn:penalized_prob} and can be easily extended to problem \eqref{eqn:penalized_prob_first}.}
\begin{equation}
\min_{\mb L, \mb S} \;\; f(\mb L, \mb S) \doteq \frac{1}{2} \norm{\PQ[\mb L + \mb S - \mb M_0]}{F}^2 + \lambda_L \norm{\mb L}{*} + \lambda_S \norm{\mb S}{1}.
\label{eqn:penalized_prob}
\end{equation}

Recently, CPCP and its close variants have been studied for different sensing operators $\mc A$ (or equivalently different subspaces $\mc Q$). In specific, \cite{chandrasekaran2011rank,candes2011robust,zhou2010stable,hsu2011robust,agarwal2012noisy} consider the case where a subset $\Omega \subseteq \set{1,2,\ldots, m} \times \set{1,2,\ldots, n}$ of the entries of $\mb M_0$ is observed. Then CPCP can be reduced to
\begin{equation}
\min_{\mb L, \mb S} \quad \frac{1}{2} \norm{\PO[\mb L + \mb S -
\mb M_0]}{F}^2 + \lambda_L \norm{\mb L}{*} + \lambda_S \norm{\mb S}{1},
\label{eqn:SPCP}
\end{equation}
where $\PO[\cdot]$ denotes the
orthogonal projection onto the linear space of matrices supported on
$\Omega$, i.e.,
$
\PO [\mb M_0] (i,j) = (\mb M_0)_{ij}$ if $(i,j) \in \Omega$ and $\PO[\mb M_0](i,j) = 0$ otherwise. \cite{wright2013compressive} studies the case where each $\mc A_k$ is an i.i.d. $\mc N(0,1)$ matrix, which is equivalent (in distribution) to saying that we choose a linear subspace $\mc Q$ uniformly at random from the set of all $p$-dimensional subspaces of $\reals^{m \times n}$ and observe $\PQ[\mb M_0]$. Accordingly, all the above provide theoretical guarantees for CPCP, under fairly mild conditions, to produce accurate estimates of $\mb L_0$ and $\PO[\mb S_0]$ (or $\mb S_0$), even when the number of measurements $p$ is substantially less than $mn$.

Inspired by these theoretical results, researchers from different fields have leveraged CPCP to solve many practical problems, including video background modeling \cite{candes2011robust}, batch image alignment \cite{Peng2012-PAMI}, face verification \cite{Zhang2013-ICCV}, photometric stereo \cite{Wu2011-ACCV}, dynamic MRI \cite{Otazo2013-pp}, topic modeling \cite{Min2010-CIKM}, latent variable graphical model learning \cite{chandrasekaran2012latent} and outlier detection and robust Principal Component Analysis \cite{candes2011robust}, to name just a few.

Living in the era of {\em big data}, most of these applications involve large datasets and high dimensional data spaces. Therefore, to fully realize the benefit of the theory, we need {\em provably convergent}  and {\em scalable} algorithms for CPCP. This has motivated much research into the development of first-order methods for problem \eqref{eqn:penalized_prob} and its variants; e.g see \cite{lin2009fast,lin2010augmented,yuan2009sparse,aybat2011fast,tao2011recovering,aybat2012efficient}. These methods, in essence, all exploit a closed-form expression for the proximal operator of the nuclear norm, which involves the singular value decompsition (SVD). Hence, the dominant cost in each iteration is computing an SVD of the same size as the input data. This is substantially more scalable than off-the-shelf interior point solvers such as SDPT3 \cite{tutuncu2003solving}. Nevertheless, the superlinear cost of each iteration has limited the practical applicability of these first-order methods to problems involving several thousands of data points and several thousands of dimensions. The need to compute a sequence of full or partial SVDs is a serious bottleneck for truly large-scale applications.

As a remedy, in this paper, we design more scalable algorithms to solve CPCP that compute only a rank-one SVD in each iteration. Our approach leverages two classical and widely studied ideas -- Frank-Wolfe iterations to handle the nuclear norm, and proximal steps to handle the $\ell_1$ norm. This turns out to be an ideal combination of techniques to solve large-scale CPCP problems. In particular, it yields algorithms that are substantially {\em more scalable} than prox-based first-order methods such as ISTA and FISTA \cite{beck2009fast}, and converge {\em much faster} in practice than a straightforward application of Frank-Wolfe.

The remainder of this paper is organized as follows. Section \ref{sec:prelim} reviews the general properties of the Frank-Wolfe algorithm, and describes several basic building blocks that we will use in our algorithms. Section \ref{sec:FW-P} and Section \ref{sec:FW-T} respectively describe how to modify the Frank-Wolfe algorithm to solve CPCP's {\em norm constrained} version
\begin{equation}
\min_{\mb L, \mb S} \;\; l(\mb L, \mb S) \doteq \frac{1}{2} \norm{\PQ[\mb L + \mb S - \mb M_0]}{F}^2  \quad \mbox{s.t.}\; \norm{\mb L}{*}\le \tau_L, \; \norm{\mb S}{1}\le \tau_S,
\label{eqn:norm_cons_prob_first}
\end{equation}
and the penalized version, i.e. problem \eqref{eqn:penalized_prob}, by incorporating proximal regularization to more effectively handle the $\ell_1$ norm. Convergence results and our implementation details are also discussed. Section \ref{sec:num} presents numerical experiments on large datasets that demonstrate the scalability of our proposed algorithms. In Section \ref{sec:conclusion}, we summarize our contributions and discuss potential future works.

\section{Preliminaries} \label{sec:prelim}

\subsection{Frank-Wolfe method} \label{subsec:FW}
The Frank-Wolfe (FW) method \cite{frank1956algorithm}, also known as the conditional gradient method \cite{levitin1966constrained},  applies to
the general problem of minimizing a differentiable convex function $h$ over a compact, convex domain $\mc D \subseteq \reals^n$:
\begin{equation} \label{eqn:general_constr}
\mbox{minimize} \quad h(\bm x)   \qquad  \mbox{subject to } \quad \bm x \in \mc D \subseteq \reals^n.
\end{equation}
Here, $\nabla h$ is assumed to be $L$-Lipschitz:
\begin{equation}
\forall \, \bm x, \, \mb y \in \mc D, \qquad \norm{\nabla h(\bm x) - \nabla h(\mb y) }{} \le L \norm{\bm x - \mb y }{}.
\end{equation}
Throughout, we let $D = \max_{\bm x,\mb y \in \mc D} \norm{\bm x - \mb y}{}$ denote the diameter of the feasible set $\mc D$.

In its simplest form, the Frank-Wolfe algorithm proceeds as follows. At each iteration $k$, we linearize the objective function $h$ about the current point $\bm x^k$:
\begin{equation}
h(\bm v) \approx h(\bm x^k) + \innerprod{ \nabla h(\bm x^k) }{\bm v - \bm x^k}.
\end{equation}
 We minimize the linearization over the feasible set $\mc D$ to obtain
 \begin{flalign}
 \bm v^k \in \arg \min_{\bm v \in \mc D} \innerprod{\nabla h(\bm x^k)}{\bm v},
 \end{flalign}
 and then take a step in the feasible descent direction $\bm v^k - \bm x^k$:
\begin{equation} \label{eqn:fixed_step_size}
\bm x^{k+1} = \bm x^k + \frac{2}{k+2} ( \bm v^k - \bm x^k ).
\end{equation}
This yields a very simple procedure, which we summarize as Algorithm \ref{alg:Frank-Wolfe}. The particular step size, $\frac{2}{k+2}$, comes from the convergence analysis of the algorithm, which we discuss in more details below.

\begin{algorithm}[tb]
   \caption{Frank-Wolfe method for problem \eqref{eqn:general_constr}}
   \label{alg:Frank-Wolfe}
\begin{algorithmic}[1]
   \STATE {\bfseries Initialization:} $\bm x^0 \in \mc D$;
   \FOR{$k=0,\; 1,\; 2, \; \ldots$}
   \STATE $\bm v^k  \in \mbox{argmin}_{\bm v\in \mc D} \innerprod{\bm v}{\nabla h(\bm x^k)}$;
   \STATE $\gamma = \frac{2}{k+2}$;
   \STATE $\bm x^{k+1} =\bm x^{k}+ \gamma (\bm v^k- \bm x^{k});$ \label{eqn:updating}
   \ENDFOR
\end{algorithmic}
\end{algorithm}

First proposed in \cite{frank1956algorithm}, FW-type methods have been
frequently revisited in different fields.
Recently, they have experienced a resurgence in statistics, machine learning and signal processing, due to their ability to yield highly scalable algorithms for optimization with structure-encouraging norms such as the $\ell_1$ norm and nuclear norm. In particular, if $\bm x$ is a matrix and $\mc D = \set{ \bm x \mid \norm{\bm x}{*} \le \beta }$ is a nuclear norm ball, the subproblem
\begin{equation}
\min_{\bm v \in \mc D}\; \innerprod{ \bm v }{\nabla h(\bm x)}
\end{equation}
can be solved using only the singular vector pair corresponding to the single leading singular value of the matrix $\nabla h(\bm x)$. Thus, at each iteration, we only have to compute a rank-one partial SVD.  This is substantially cheaper than the full/partial SVD exploited in proximal methods \cite{jaggi2010simple,harchaoui2013conditional}. We recommend \cite{jaggi2013revisiting} as a comprehensive survey of the latest developments in FW-type methods.

\begin{algorithm}[hb]
   \caption{Frank-Wolfe method for problem \eqref{eqn:general_constr} with general updating scheme}
\begin{algorithmic}[1]
\label{alg:general_frank_wolfe}
   \STATE {\bfseries Initialization:} $\bm x^0 \in \mc D$;
   \FOR{$k=0,\; 1,\; 2, \; \ldots$}
   \STATE $\bm v^k \in \mbox{argmin}_{\bm v\in \mc D} \innerprod{\bm v}{\nabla h(\bm x^k)}$;
   \STATE $\gamma = \frac{2}{k+2}$ ;
   \STATE Update $\bm x^{k+1}$ to some point in $\mc D$ such that $h(\bm x^{k+1}) \le h(\bm x^{k}+\gamma (\bm v^k- \bm x^{k}));$
   %\STATE $\bm x^{k+1} =\bm x^{k}+\gamma (\bm v^k- \bm x^{k});$
   \ENDFOR
\end{algorithmic}
\end{algorithm}

In the past five decades, numerous variants of Algorithm \ref{alg:Frank-Wolfe} have been proposed and implemented. Many modify Algorithm \ref{alg:Frank-Wolfe} by replacing the simple updating rule \eqref{eqn:fixed_step_size}
with more sophisticated schemes, e.g.,
\begin{equation} \label{eqn:exact_line}
\bm x^{k+1} \; \in \;  \arg \min_{\bm x} \; h(\bm x) \quad   \mbox{s.t.} \;\; \bm x \in \mbox{conv}\{\bm x^k, \; \bm v^k\}
\end{equation}
or
\begin{equation} \label{eqn:exact_line_more}
\bm x^{k+1} \; \in \; \arg \min_{\bm x} \; h(\bm x) \quad   \mbox{s.t.} \;\; \bm x \in \mbox{conv}\{\bm x^k, \; \bm v^k, \; \bm v^{k-1}, \; \ldots, \; \bm v^{k-j} \}.
\end{equation}
The convergence of these schemes can be analyzed simultaneously, using the fact that they produce iterates $\mb x^{k+1}$ whose objective is no greater than that produced by the original Frank-Wolfe update scheme:
$$
h(\mb x^{k+1}) \le h(\mb x^k + \gamma(\mb v^k - \mb x^k)).
$$
 Algorithm \ref{alg:general_frank_wolfe} states a general version of Frank-Wolfe, whose update is only required to satisfy this relationship. It includes as special cases the updating rules \eqref{eqn:fixed_step_size}, \eqref{eqn:exact_line} and \eqref{eqn:exact_line_more}. This flexibility will be crucial for effectively handling the sparse structure in the CPCP problems \eqref{eqn:penalized_prob} and \eqref{eqn:norm_cons_prob_first}.

The convergence of Algorithm \ref{alg:general_frank_wolfe} can be proved using well-established techniques
%for analyzing the convergence of FW-type algorithms
 \cite{harchaoui2013conditional,jaggi2013revisiting,demʹi︠a︡nov1970approximate,Dunn1978432,Patriksson93,zhang2003sequential,Clarkson:2010:CSG:1824777.1824783,freund2014new}. Using these ideas, one can show that it converges at a rate of $O(1/k)$ in function value:
\begin{theorem}\label{thm:fw_primal_conv}
Let $\bm x^\star$ be an optimal solution to \eqref{eqn:general_constr}. For $\{\bm x^k\}$ generated by Algorithm $\ref{alg:general_frank_wolfe}$, we have for $k = 0,\; 1, \;2, \;\ldots,$
%$$
%h(\bm x^{k+1}) - h(\bm x^\star) \le \frac{k}{k+2}\left(h(\bm x^k) - h(\bm x^\star)\right) + \frac{LD^2}{2}\left(\frac{2}{k+2}\right)^2,
%$$
%which by Lemma \ref{lem:recurrence} implies that for $k = 1,\; 2, \;3, \;\ldots,$
\begin{equation}\label{eqn:primal_conv}
h(\bm x^{k})-h(\bm x^\star) \le \frac{2LD^2}{k+2}.
\end{equation}
\end{theorem}
\begin{proof}
For $k = 0,\; 1, \;2, \;\ldots,$ we have
\begin{eqnarray}
h(\bm x^{k+1}) & \le & h(\bm x^k + \gamma (\bm v^k - \bm x^k)) \nonumber\\
&\le& h(\bm x^k) + \gamma \innerprod{\nabla h(\bm x^k)}{\bm v^k-\bm x^k}+\frac{L\gamma^2}{2}\norm{\bm v^k - \bm x^k}{}^2  \nonumber\\
&\le& h(\bm x^k) + \gamma \innerprod{\nabla h(\bm x^k)}{\bm v^k-\bm x^k}+ \frac{\gamma^2 LD^2}{2} \label{eqn:fw_pf_prim_conv_first} \\
&\le&  h(\bm x^k) + \gamma \innerprod{\nabla h(\bm x^k)}{\bm x^\star-\bm x^k}+ \frac{\gamma^2 LD^2}{2}    \nonumber\\
&\le& h(\bm x^k)  + \gamma(h(\bm x^\star)-h(\bm x^k))+\frac{\gamma^2 LD^2}{2}, \label{eqn:fw_pf_prim_conv}
\end{eqnarray}
where the second inequality holds since $\nabla h(\cdot)$ is $L$-Lipschitz continuous; the third line follows because $D$ is the diameter for the feasible set $\mc D$; the fourth inequality follows from $\bm v^k \in \mbox{argmin}_{\bm v\in \mc D} \innerprod{\bm v}{\nabla h(\bm x^k)}$ and $\bm x^\star \in \mc D$; the last one holds since $h(\cdot)$ is convex.

Rearranging terms in \eqref{eqn:fw_pf_prim_conv}, one obtains that for $k = 0,\; 1, \;2, \;\ldots,$
\begin{equation} \label{eqn:fw_pf_prim_conv_2}
h(\bm x^{k+1}) - h(\bm x^\star) \le (1-\gamma)\left(h(\bm x^k) - h(\bm x^\star)\right) + \frac{\gamma^2 LD^2}{2}.
\end{equation}
Therefore, by mathematical induction, it can be verified that
$$
h(\bm x^{k})-h(\bm x^\star) \le \frac{2LD^2}{k+2}, \quad \mbox{for} \;\; k = 1,\; 2, \;3, \;\ldots.
$$
\end{proof}
\begin{remark}
Note that the constant in the rate of convergence depends on the Lipschitz constant $L$ of $h$ and the diameter $\mc D$.
\end{remark}
%This result was perhaps first derived by \cite{demʹi︠a︡nov1970approximate}. For completeness, we provide a proof of Theorem \ref{thm:fw_primal_conv} in the Appendix.

While Theorem \ref{thm:fw_primal_conv} guarantees that Algorithm \ref{alg:general_frank_wolfe} converges at a rate of $O(1/k)$, in practice it is useful to have a more precise bound on the suboptimality at iterate $k$. The surrogate duality gap
\begin{equation}\label{eqn:dual_gap}
d(\bm x^k) = \innerprod{\bm x^k - \bm v^k}{\nabla h(\bm x^k)},
\end{equation}
provides a useful upper bound on the suboptimality $h(\bm x^k) - h(\bm x^\star)$ :
\begin{flalign}
   h(\bm x^k) - h(\bm x^\star)  &\le -\innerprod{\bm x^\star - \bm x^k}{\nabla h(\bm x^k)} \nonumber \\
                                     &\le -\min_{\bm v} \innerprod{\bm v - \bm x^k}{\nabla h(\bm x^k)}
                                      =  \innerprod{\bm x^k - \bm v^k}{\nabla h(\bm x^k)}
                                      = d(\bm x^k).
\end{flalign}
This was first proposed in \cite{frank1956algorithm} and later \cite{jaggi2013revisiting} showed that $d(\bm x^k) = O(1/k)$.
Next, we provide a refinement of this result, using ideas from \cite{jaggi2013revisiting,Clarkson:2010:CSG:1824777.1824783}:
%In the Appendix, we prove the following refinement of this result, using ideas from \cite{jaggi2013revisiting,Clarkson:2010:CSG:1824777.1824783}:
\begin{theorem} \label{thm:fw_dual}
Let $\{\bm x^k\}$ be the sequence generated by Algorithm $\ref{alg:general_frank_wolfe}$. Then for any $K\ge 1$, there exists $1\le \tilde k \le K$ such that
\begin{equation}
d(\bm x^{\tilde k}) \le \frac{6LD^2}{K+2}.
\end{equation}
\end{theorem}
\begin{proof}
For notational convenience, we denote $h^k \doteq h(\bm x^k)$, $\Delta^k \doteq h(\bm x^k) - h(\bm x^\star)$, $d^k \doteq d(\bm x^k)$, $C \doteq 2LD^2$, $B \doteq K+2$, $\hat k \doteq \lceil \frac{1}{2}B\rceil-1$, $\mu \doteq \lceil \frac{1}{2}B\rceil/B$.

Suppose on the contrary that
\begin{equation}\label{eqn:fw_pf_dual_conv_2}
d^k > \frac{3C}{B},  \quad \mbox{for all} \;\;k \in \set{\lceil \frac{1}{2}B\rceil-1, \; \lceil \frac{1}{2}B\rceil, \; \ldots, \;K}.
\end{equation}
From \eqref{eqn:fw_pf_prim_conv_first}, we know that for any $k\ge 1$
\begin{equation}\label{eqn:fw_pf_dual_conv_1}
\Delta^{k+1} \le \Delta^k + \gamma\innerprod{\nabla h(\bm x^k)}{\bm v^k-\bm x^k}+\frac{\gamma^2 LD^2}{2} = \Delta^k - \frac{2d^k}{k+2}  + \frac{C}{(k+2)^2}.
\end{equation}
Therefore, by using \eqref{eqn:fw_pf_dual_conv_1} repeatedly, one has
\begin{eqnarray}
\Delta^{K+1} &\le& \Delta^{\hat k} - \sum_{k = \hat k}^K \frac{2d^k}{k+2}  + \sum_{k = \hat k}^K \frac{C}{(k+2)^2} \nonumber\\
&<& \Delta^{\hat k} - \frac{6C}{B} \sum_{k = \hat k}^K \frac{1}{k+2} + C \sum_{k = \hat k}^K \frac{1}{(k+2)^2} \nonumber\\
&=& \Delta^{\hat k} - \frac{6C}{B} \sum_{k = \hat k+2}^B \frac{1}{k}+ C \sum_{k = \hat k+2}^B \frac{1}{k^2} \nonumber\\
&\le& \nonumber \frac{C}{\mu B} - \frac{6C}{B}\cdot\frac{B-\hat k -1}{B} + C\cdot \frac{B-\hat k -1}{B(\hat k+1)}\\
&=& \frac{C}{\mu B} - \frac{6C}{B}(1-\mu)+\frac{C}{B}\frac{1-\mu}{\mu} \nonumber\\
&=& \frac{C}{\mu B}\left(2 - 6\mu(1-\mu) - \mu  \right) \label{eqn:fw_pf_dual_conv}
\end{eqnarray}
where the second line is due to our assumption \eqref{eqn:fw_pf_dual_conv_2};
the fourth line holds since $\Delta^{\hat k} \le \frac{C}{\hat k+2}$ by Theorem 1, and $\sum_{k=a}^b \frac{1}{k^2} \le \frac{b-a+1}{b(a-1)}$ for any $b \ge  a > 1$.

Now define $\phi(x) = 2 - 6x(1-x) - x$. Clearly $\phi(\cdot)$ is convex. Since $\phi(\frac{1}{2}) = \phi(\frac{2}{3}) = 0$, we have $\phi(x) \le 0$ for any $x \in [\frac{1}{2},\frac{2}{3}]$. As $\mu = \lceil \frac{1}{2}B\rceil/B \in [\frac{1}{2},\frac{2}{3}]$, from \eqref{eqn:fw_pf_dual_conv}, we have
$$
\Delta^{K+1} = h(\bm x^{K+1}) - h(\bm x^\star) < \frac{C}{\mu B} \phi(\mu) \le 0,
$$
which is a contradiction.
\end{proof}
\begin{remark}
The convergence rate for the duality gap matches the one for $h(\bm x^k) - h(\bm x^\star)$ (see \eqref{eqn:primal_conv}), which suggests that the upper bound $d(\bm x^k)$ can serve as a practical stopping criterion.
\end{remark}

For our problem, the main computational burden in Algorithms \ref{alg:Frank-Wolfe} and \ref{alg:general_frank_wolfe} will be solving the linear subproblem $\min_{\bm v\in \mc D} \innerprod{\bm v}{\nabla h(\bm x^k)}$, \footnote{In some situations, we can significantly reduce this cost by solving this problem inexactly \cite{Dunn1978432,jaggi2013revisiting}. Our algorithms and results can also tolerate inexact step calculations; we omit the discussion here for simplicity.} i.e. minimizing linear functions over the unit balls for $\norm{\cdot}{*}$ and $\norm{\cdot}{1}$. Fortunately, both of these operations have simple closed-form solutions, which we will describe in the next section.

\subsection{Optimization oracles}\label{subsec:oracles}
We now describe several optimization oracles involving the
$\ell_1$ norm and the nuclear norm, which serve as the main building blocks for our methods. These oracles have computational costs that are (essentially) linear in the size of the input.
\paragraph{\textbf{Minimizing a linear function over the nuclear norm ball}} Since the dual norm of the nuclear norm is the operator norm, i.e., $\norm{\mb Y}{}= \max_{\norm{\mb X}{*} \le 1}
  \innerprod{\mb Y}{\mb X}$, the optimization problem
  \begin{equation}
     \mbox{minimize}_{\mb X} \;\; \<\mb Y, \mb X \>  \qquad \mbox{subject to} \; \norm{\mb X}{*} \le 1
  \end{equation}
  has optimal value $-\norm{\mb Y}{}$. One minimizer is the rank-one matrix $\mb X^\star = -\bm
  u \bm v^\top$, where $\bm u$ and $\bm v$ are the left- and right-
  singular vectors corresponding to the leading singular value of $\mb Y$, and can be efficiently computed (e.g. using power method).

  \paragraph{\textbf{Minimizing a linear function over the $\ell_1$ ball}}   Since the dual norm of the $\ell_1$ norm is the $\ell_\infty$ norm, i.e.,  $\norm{\mb Y}{\infty}:= \max_{(i,j)} |Y_{ij}| = \max_{\norm{\mb
  X}{1} \le 1} \innerprod{\mb Y}{\mb X}$, the optimization problem
  \begin{equation} \label{eqn:dual_norm_l1}
     \mbox{minimize}_{\mb X} \;\; \<\mb Y, \mb X \>  \qquad \mbox{subject to} \; \norm{\mb X}{1} \le 1
  \end{equation}
  has optimal value $-\norm{\mb Y}{\infty}$. One minimizer is the one-sparse matrix 
  \[
  \mb X^\star = - \mbox{sgn}(Y_{i^\star j^\star}) \bm e_{i^\star} \bm e_{j^\star}^{\t},
  \]
  where $(i^\star, j^\star)\in \arg \max_{(i,j)} |Y_{ij}|$; i.e. $\mb X^\star$ has exactly one nonzero element.

  \paragraph{\textbf{Projection onto the $\ell_1$-ball}} To effectively handle the sparse term in the norm constrained problem \eqref{eqn:norm_cons_prob_first}, we will need to modify the Frank-Wolfe algorithm by incorporating additional projection steps. For any $\mb Y \in \reals^{m \times n }$ and $\beta > 0$, the projection onto the $\ell_1$-ball:
    \begin{equation} \label{eqn:l_1_proj_oracle}
    \Proj_{\norm{\cdot}{1}\le \beta}[\mb Y]  =  \arg \min_{\norm{\mb X}{1} \le \beta} \;\; \frac{1}{2} \norm{\mb X- \mb Y}{F}^2,
  \end{equation}
  can be easily solved with $O\left(mn (\log m+\log n)\right)$ cost \cite{duchi2008efficient}. Moreover, a divide and conquer algorithm, achieving
  linear cost in expectation to solve \eqref{eqn:l_1_proj_oracle}, has also been proposed in \cite{duchi2008efficient}.

  \paragraph{\textbf{Proximal mapping of $\ell_1$ norm}} To effectively handle the sparse term arising in problem \eqref{eqn:penalized_prob}, we will need to modify the Frank-Wolfe algorithm by incorporating additional proximal steps. For any $\mb Y \in \reals^{m \times n }$ and $\lambda > 0$, the proximal mapping of $\ell_1$ norm has the following closed-form expression
  \begin{equation}
   \mc T_{\lambda}[\mb Y] \; = \; \arg \min_{\mb X \in \reals^{m\times n}} \;\; \frac{1}{2} \norm{\mb X- \mb Y}{F}^2 + \lambda \norm{\mb X}{1},
  \end{equation}
  where $\mc T_{\lambda}: \reals \to \reals$ denotes the soft-thresholding operator $\mc
  T_{\lambda}(x) = \mbox{sgn}(x) \max\{|x|- \lambda, 0\}$, and extension to matrices is obtained by applying the scalar operator $\mc T_\lambda(\cdot)$ to each element.

\section{FW-P Method for Norm Constrained Problem}\label{sec:FW-P}
In this section, we develop scalable algorithms for the norm-constrained compressive principal component pursuit problem,
\begin{equation}
 \min_{\mb L, \mb S} \; l(\mb L, \mb S)=\frac{1}{2} \norm{\PQ [\mb L+\mb S - \mb M]}{F}^2 \;\;\;  \mbox{s.t.} \;\; \; \norm{\mb L}{*} \le \tau_L, \; \norm{\mb S}{1} \le \tau_S.
\label{eqn:norm_cons_prob_2}
\end{equation}
We first describe a straightforward application of the Frank-Wolfe method to this problem. We will see that although it has relatively cheap iterations, it converges very slowly on typical numerical examples, because it only makes a one-sparse update to the sparse term $\mb S$ at a time. We will show how to remedy this problem by augmenting the FW iteration with an additional proximal step (essentially a projected gradient step) in each iteration, yielding a new algorithm which updates $\mb S$ much more efficiently. Because it combines Frank-Wolfe and projection steps, we will call this new algorithm Frank-Wolfe-Projection (FW-P).

\paragraph{Properties of the objective and constraints.} To apply Frank-Wolfe to \eqref{eqn:norm_cons_prob_2}, we first note that the objective $l(\mb L,\mb S)$ in \eqref{eqn:norm_cons_prob_2} is differentiable, with
\begin{eqnarray}
\nabla_{\mb L} l(\mb L, \mb S) &=& \PQ[ \mb L + \mb S - \mb M ] \label{eqn:grad-f-L}\\
\nabla_{\mb S} l(\mb L, \mb S) &=& \PQ[ \mb L + \mb S - \mb M ]. \label{eqn:grad-f-S}
\end{eqnarray}
Moreover, the following lemma shows that the gradient map $\nabla l(\mb L, \mb S) = (\nabla_{\mb L} l, \nabla_{\mb S} l)$ is 2-Lipschitz:
\begin{lemma}\label{lem:lipschitz_p} For all $(\mb L,\mb S)$ and $(\mb L',\mb S')$, we have  $\norm{\nabla l(\mb L, \mb S) - \nabla l(\mb L',\mb S')}{F} \le 2 \norm{ (\mb L,\mb S) - (\mb L',\mb S') }{F}$.
\end{lemma}
\begin{proof} From \eqref{eqn:grad-f-L} and \eqref{eqn:grad-f-S}, we have
\begin{eqnarray*}
\norm{ \nabla l(\mb L,\mb S) - \nabla l(\mb L',\mb S') }{F}^2 &=& 2 \norm{ \PQ[ \mb L + \mb S - \mb M ] - \PQ[ \mb L' + \mb S' - \mb M ] }{F}^2 \\
&=& 2 \norm{ \PQ[ \mb L + \mb S ] - \PQ[ \mb L' + \mb S' ] }{F}^2 \\
&\le& 2 \norm{ \mb L + \mb S - \mb L' - \mb S' }{F}^2 \\
&\le& 4 \norm{\mb L - \mb L'}{F}^2 + 4 \norm{\mb S - \mb S'}{F}^2 \\
&=& 4 \norm{ (\mb L,\mb S) - (\mb L',\mb S') }{F}^2,
\end{eqnarray*}
which implies the result.
\end{proof}

\noindent The feasible set in \eqref{eqn:norm_cons_prob_2} is compact. The following lemma bounds its diameter $D$:
\begin{lemma}\label{lem:diameter_p} The feasible set $\mc D = \set{ (\mb L,\mb S) \mid \norm{\mb L}{*} \le \tau_L, \; \norm{\mb S}{1} \le \tau_S }$ has diameter $D \le 2 \sqrt{ \tau_L^2 + \tau_S^2 }$.
\end{lemma}
\begin{proof}
For any $\mb Z = (\mb L, \mb S)$ and $\mb Z' = (\mb L', \mb S') \in \mc D$,
\begin{flalign} %\label{eqn:nc_diameter}
\norm{\mb Z - \mb Z'}{F}^2 &= \norm{\mb L-\mb L'}{F}^2+\norm{\mb S-\mb S'}{F}^2 \le (\norm{\mb L}{F}+\norm{\mb L'}{F})^2+(\norm{\mb S}{F}+\norm{\mb S'}{F})^2 \nonumber \\
&\le (\norm{\mb L}{*}+\norm{\mb L'}{*})^2+(\norm{\mb S}{1}+\norm{\mb S'}{1})^2 \le 4\tau_L^2 + 4\tau_S^2.
\end{flalign}
\end{proof}
%In this section, we describe our FW-P method for the norm-constrained
%problem \eqref{eqn:norm_cons_prob}, which consists of an FW step and a
%projection (P) step in each iteration. We also establish an $O(1/k)$ convergence result for the FW-P method.

\subsection{Frank-Wolfe for problem \eqref{eqn:norm_cons_prob_2}}

Since \eqref{eqn:norm_cons_prob_2} asks us to minimize a convex, differentiable function with Lipschitz gradient over a compact convex domain, the Frank-Wolfe method in Algorithm \ref{alg:Frank-Wolfe} applies. It generates a sequence of iterates $\bm x^k = (\mb L^k, \mb S^k)$. Using the expression for the gradient in \eqref{eqn:grad-f-L}-\eqref{eqn:grad-f-S}, at each iteration, the step direction $\bm v^k = (\mb V_L^k, \mb V_S^k)$ is generated by solving the linearized subproblem
\begin{eqnarray}\label{eqn:nc_v}
\left( \begin{array}{ccc} \mb V_L^k \\ \mb V_S^k \end{array} \right) \in \arg \min& \left\<\left( \begin{array}{ccc} \PQ[\mb L^k + \mb S^k - \mb M] \\ \PQ[\mb L^k + \mb S^k - \mb M] \end{array} \right),
\left( \begin{array}{ccc} \mb V_L \\ \mb V_S \end{array} \right)\right\> \label{eqn:subprob_1}\\
        \mbox{s.t.}& \norm{\mb V_L}{*} \le \tau_L, \;\; \norm{\mb V_S}{1} \le \tau_S,  \nonumber
\end{eqnarray}
which decouples into two independent subproblems:
\begin{eqnarray*}
\mb V_L^k &\in& \arg \min_{\norm{\mb V_L}{*} \le \tau_L}  \< \PQ[\mb L^k + \mb S^k - \mb M],\;  \mb V_L\>,\\
\bm V_S^k &\in& \arg \min_{\norm{\mb V_S}{1} \le \tau_S}  \< \PQ[\mb L^k + \mb S^k - \mb M],\;  \mb V_S\>.
\end{eqnarray*}
These subproblems can be easily solved by exploiting the linear optimization
oracles introduced in Section \ref{subsec:oracles}. In particular,
\begin{eqnarray}
\mb V_L^k &=& -\tau_L \bm u^k (\bm v^k)^{\t}, \label{eqn:V-step} \\
\mb V_S^k &=& -\tau_S \cdot \delta^k_{i^\star j^\star}\cdot \bm e^k_{i^\star} \mb (\bm e^k_{j^\star})^{\t}, \label{eqn:S-step}
\end{eqnarray}
where $\mb u^k$ and $\mb v^k$ are leading left- and right- singular vectors of $\PQ[ \mb L^k + \mb S^k - \mb M]$ and $(i^\star,j^\star)$ is the  of the largest element of $\PQ[\mb L^k + \mb S^k - \mb M ]$ in magnitude and $\delta^k_{ij}:=\mbox{sgn}\left[ \left(\mc P_{\mc Q} \left[ \mb L^k + \mb S^k - \mb M \right] \right)_{ij} \right]$. Algorithm \ref{alg:FW_naive_constr} gives the Frank-Wolfe method specialized to problem \eqref{eqn:norm_cons_prob_2}.

\begin{algorithm}[h]
   \caption{Frank-Wolfe method for problem \eqref{eqn:norm_cons_prob_2}}
   \label{alg:FW_naive_constr}
\begin{algorithmic}[1]
   \STATE {\bfseries Initialization:} $\mb L^0 = \mb S^0 = \mb 0;$
   \FOR{$k=0,\; 1,\; 2, \; \cdots$}
   \STATE $\mb D_L^k \in \arg \min_{\norm{\mb D_L}{*}\le 1} \< \PQ[\mb L^k + \mb S^k - \mb M], \; \mb D_L\>$; $\mb V_L^k = \tau_L \mb D_L^k$;
   \STATE $\mb D_S^k \in \arg \min_{\norm{\mb D_S}{1}\le 1} \< \PQ[\mb L^k + \mb S^k - \mb M], \; \mb D_S\>$; $\mb V_S^k = \tau_S \mb D_S^k$;
   \STATE $\gamma = \frac{2}{k+2}$;
   \STATE$\mb L^{k+1} = \mb L^k + \gamma(\mb V_L^k - \mb L^k);$
   \STATE$\mb S^{k+1} = \mb S^k + \gamma(\mb V_S^k - \mb S^k);$
   \ENDFOR
\end{algorithmic}
\end{algorithm}

The major advantage of Algorithm \ref{alg:FW_naive_constr} lies in the simplicity of the update rules \eqref{eqn:V-step}-\eqref{eqn:S-step}. Both have closed form, and both can be computed in time (essentially) linear in the size of the input. Because $\mb V_L^k$ is rank-one, the algorithm can be viewed as performing a sequence of rank one updates.

The major disadvantage of Algorithm \ref{alg:FW_naive_constr} is that $\mb S$ has only a one-sparse update at each iteration, since $\mb V_S^k = -\tau_S \bm e^k_{i^\star} \mb (\bm e^k_{j^\star})^{\t}$ has only one nonzero entry. This is a significant disadvantage in practice, as the optimal $\mb S^\star$ may have a relatively large number of nonzero entries. Indeed, in theory, the CPCP relaxation works even when a constant fraction of the entries in $\mb S_0$ are nonzero. In applications such as foreground-background separation, the number of nonzero entries in the target sparse term can be quite large. The dashed curves in Figure \ref{fig:synthetic} show the effect of this on the practical convergence of the algorithm, on a simulated example of size $1,000 \times 1,000$, in which about $1\%$ of the entries in the target sparse matrix $\mb S_0$ are nonzero. As shown, the progress is quite slow.

\begin{figure}[t]
\centerline{
\begin{minipage}{2.75in}
\centerline{\includegraphics[width= 5.1in]{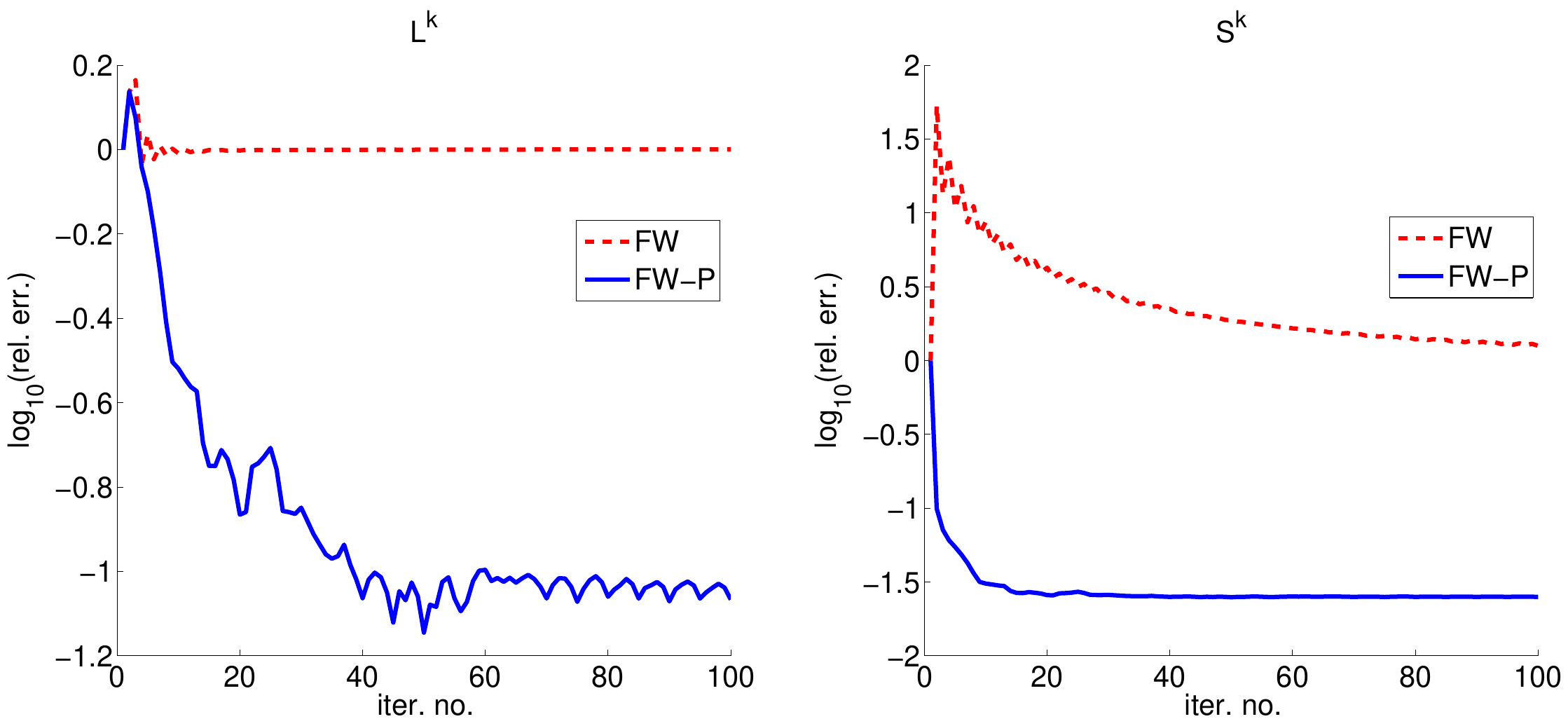}}
\end{minipage}
}
\caption{\textbf{Comparisons between Algorithms \ref{alg:FW_naive_constr} and \ref{alg:Frank-Wolfe_Proj} for problem \eqref{eqn:norm_cons_prob_2} on synthetic data.} The data are generated in
Matlab as
$\mathtt{m = 1000; \; n = 1000; \; r = 5;}$ $\mathtt{L_0 = randn(m,r) * randn(r,n);}$ $\mathtt{Omega = ones(m,n);}$
$\mathtt{S_0 = 100*randn(m,n).*(rand(m,n) < 0.01);}$ $\mathtt{M = L_0 + S_0 + randn(m,n);}$ $\mathtt{\tau_L = norm\_nuc(L_0);}$
 $\mathtt{\tau_S = norm(vec(S_0),1);}$
The left figure plots $\log_{10}(\norm{\mb L^k - \mb L_0}{F}/\norm{\mb L_0}{F})$ versus
the iteration number $k$. The right figure plots $\log_{10} (\norm{\mb S^k - \mb S_0}{F}/\norm{\mb S_0}{F})$ versus
$k$. The FW-P method is clearly more efficient than the straightforward FW method in recovering $\mb L_0$ and $\mb S_0$.
}
\label{fig:synthetic}
\end{figure}

\subsection{FW-P algorithm: combining Frank-Wolfe and projected gradient}
To overcome the drawback of the naive Frank-Wolfe algorithm described above, we propose incorporating an additional gradient projection step after each Frank-Wolfe update. This additional step updates the sparse term $\mb S$ only, with the goal of accelerating convergence in these variables. At iteration $k$, let $(\mb L^{k+1/2},\mb S^{k+1/2})$ be the result produced by Frank-Wolfe. To produce the next iterate, we retain the low rank term $\mb L^{k+1/2}$, but set
\begin{eqnarray}
\mb S^{k+1} &=& \mc P_{\norm{\cdot}{1} \le \tau_S}\left[ \, \mb S^{k+\frac{1}{2}} - \nabla_{\mb S} l(\mb L^{k+\frac{1}{2}},\mb S^{k+\frac{1}{2}} ) \, \right] \\
 &=& \Proj_{\norm{\cdot}{1} \le \tau_S} \left[\,\mb S^{k+\frac{1}{2}} - \PQ [\mb L^{k+\frac{1}{2}}+\mb S^{k+\frac{1}{2}}-\mb M] \, \right]; \label{eqn:projection}
\end{eqnarray}
i.e.  we simply take an additional projected gradient step in the sparse term $\mb S$. The resulting algorithm is presented as Algorithm \ref{alg:Frank-Wolfe_Proj} below. We call this method the FW-P algorithm, as it combines Frank-Wolfe steps and projections. In Figure \ref{fig:synthetic}, we compare Algorithms \ref{alg:FW_naive_constr} and \ref{alg:Frank-Wolfe_Proj} on synthetic data. In this example, the FW-P method is clearly more efficient in recovering $\mb L_0$ and $\mb S_0$.

\begin{algorithm}[h]
   \caption{FW-P method for problem \eqref{eqn:norm_cons_prob_2}}
   \label{alg:Frank-Wolfe_Proj}
\begin{algorithmic}[1]
   \STATE {\bfseries Initialization:} $\mb L^0 = \mb S^0 = \mb 0;$
   \FOR{$k=0,\; 1,\; 2, \; \cdots$}
   \STATE $\mb D_L^k \in \arg \min_{\norm{\mb D_L}{*}\le 1} \< \PQ[\mb L^k + \mb S^k - \mb M], \; \mb D_L\>$; $\mb V_L^k = \tau_L \mb D_L^k$;
   \STATE $\mb D_S^k \in \arg \min_{\norm{\mb D_S}{1}\le 1} \< \PQ[\mb L^k + \mb S^k - \mb M], \; \mb D_S\>$; $\mb V_S^k = \tau_S \mb D_S^k$;
   \STATE $\gamma = \frac{2}{k+2}$;
   \STATE $\mb L^{k+\frac{1}{2}} = \mb L^k + \gamma(\mb V_L^k - \mb L^k);$
   \STATE $\mb S^{k+\frac{1}{2}} = \mb S^k + \gamma(\mb V_S^k - \mb S^k);$
   \STATE $\mb S^{k+1} = \Proj_{\norm{\cdot}{1} \le \tau_S} \big[\mb S^{k+\frac{1}{2}} - \PQ [\mb L^{k+\frac{1}{2}}+\mb S^{k+\frac{1}{2}}-\mb M] \big];$ \label{eqn:fw-p_proj_step}
   \STATE $\mb L^{k+1} = \mb L^{k+\frac{1}{2}};$
   \ENDFOR
\end{algorithmic}
\end{algorithm}

The convergence of Algorithm \ref{alg:Frank-Wolfe_Proj} can be analyzed by recognizing it as a specific instance of the generalized Frank-Wolfe iteration in Algorithm \ref{alg:general_frank_wolfe}. This projection step \eqref{eqn:projection} can be
regarded as a proximal step to set $\mb S^{k+1}$ as
\begin{flalign*}
\arg \min_{\norm{\mb S}{1}\le \tau_S} \quad \hat l^{k+\frac{1}{2}}(\mb S):= &l(\mb L^{k+\frac{1}{2}}, \mb S^{k+\frac{1}{2}})+ \\
&\qquad \<\nabla_{\mb S} l(\mb L^{k+\frac{1}{2}}, \mb S^{k+\frac{1}{2}}), \mb S - \mb S^{k+\frac{1}{2}} \>
+ \frac{1}{2} \norm{\mb S - \mb S^{k+\frac{1}{2}}}{F}^2.
\end{flalign*}
It can then be easily verified that
%$\hat l^{k+\frac{1}{2}}(\cdot)$ is a majorization function over $l(\mb L^{k+\frac{1}{2}}, \cdot)$ at point $\mb S^{k+\frac{1}{2}}$, i.e.,
\begin{equation}
\hat l^{k+\frac{1}{2}}(\mb S^{k+\frac{1}{2}}) = l(\mb L^{k+\frac{1}{2}}, \mb S^{k+\frac{1}{2}}), \quad  \mbox{and } \quad
\hat l^{k+\frac{1}{2}}(\mb S) \ge l(\mb L^{k+\frac{1}{2}}, \mb S) \; \; \mbox{for any } \mb S,
\end{equation}
since $\nabla_{\mb S} l(\mb L, \mb S)$ is 1-Lipschitz.
This implies that the FW-P algorithm chooses a next iterate whose objective is no worse than that produced by the Frank-Wolfe step:
\begin{eqnarray}\nonumber
l(\mb L^{k+1},\mb S^{k+1}) = l(\mb L^{k+\frac{1}{2}},\mb S^{k+1})\le \hat{l}^{k+\frac{1}{2}}(\mb S^{k+1}) \le  \hat{l}^{k+\frac{1}{2}}(\mb S^{k+\frac{1}{2}}) = l(\mb L^{k+\frac{1}{2}},\mb S^{k+\frac{1}{2}}).
\end{eqnarray}
This is precisely the property that is required to invoke
Algorithm \ref{alg:general_frank_wolfe} and Theorems \ref{thm:fw_primal_conv} and \ref{thm:fw_dual}. Using Lemmas \ref{lem:lipschitz} and \ref{lem:diameter} to estimate the Lipschitz constant of $\nabla l$ and the diameter of $\mc D$, we obtain the following result, which shows that FW-P retains the $O(1/k)$ convergence rate of the original FW method:

\begin{theorem} \label{thm:nc_main}
Let $l^\star$ be the optimal value to problem \eqref{eqn:norm_cons_prob_2}, $\bm x^k = (\mb L^k, \mb S^k)$ and $\bm v^k = (\mb V_L^k, \mb V_S^k)$ be the sequence produced by Algorithm \ref{alg:Frank-Wolfe_Proj}. Then we have
\begin{equation}
l(\mb L^k, \mb S^k) - l^\star \le \frac{16(\tau_L^2+\tau_S^2)}{k+2}.
\end{equation}
Moreover, for any $K\ge 1$, there exists $1\le \tilde k \le K$ such that the surrogate duality gap (defined in \eqref{eqn:dual_gap}) satisfies
\begin{equation}
d(\bm x^{\tilde k}) = \innerprod{\bm x^{\tilde k} - \bm v^{\tilde k}}{\nabla l(\bm x^{\tilde k})} \le \frac{48(\tau_L^2+\tau_S^2)}{K+2}.
\end{equation}
\end{theorem}
\begin{proof}
Substituting $L = 2$ (Lemma \ref{lem:lipschitz_p}) and $D \le 2 \sqrt{\tau_L^2 + \tau_S^2}$ (Lemma \ref{lem:diameter_p}) into Theorems \ref{thm:fw_primal_conv} and \ref{thm:fw_dual}, we can easily obtain the above result.
\end{proof}

\section{FW-T Method for Penalized Problem} \label{sec:FW-T} \label{sec:SPCP}
In this section, we develop a scalable algorithm for the penalized version of the CPCP problem,
\begin{eqnarray}
 \min_{\mb L, \mb S} \quad f(\mb L, \mb S) \doteq \frac{1}{2} \norm{\PQ [\mb L+\mb S - \mb M]}{F}^2 + \lambda_L \norm{\mb L}{*} + \lambda_S \norm{\mb S}{1}.
\label{eqn:penalized_prob_main}
\end{eqnarray}
In Section \ref{subsec:reformulation}, we reformulate problem \eqref{eqn:penalized_prob_main} into the form of \eqref{eqn:general_constr} so that the Frank-Wolfe method can be applied. In Section \ref{sec:FW-penalized}, we apply the Frank-Wolfe method directly to the reformulated problem, achieving linear per-iteration cost and $O(1/k)$ convergence in function value. However, because it updates the sparse term one element at a time, it converges very slowly on typical numerical examples. In Section \ref{sec:FW-T}, we introduce our FW-T method, which resolves this issue. Our FW-T method essentially exploits the Frank-Wolfe step to handle the nuclear norm and a proximal gradient step to handle the $\ell_1$-norm, while keeping iteration cost low and retaining convergence guarantees.

\subsection{Reformulation as smooth, constrained optimization}\label{subsec:reformulation}
Note that problem \eqref{eqn:penalized_prob_main} has a non-differentiable objective function and an unbounded feasible set. To apply the Frank-Wolfe method, we exploit a two-step reformulation to transform \eqref{eqn:penalized_prob_main} into the form of \eqref{eqn:general_constr}. First, we borrow ideas from \cite{harchaoui2013conditional} and work with the epigraph reformulation of \eqref{eqn:penalized_prob_main},
\begin{eqnarray} \label{eqn:penal_prob_epi}
&\min   \quad    &  g(\mb L, \mb S, t_L, t_S) \doteq \frac{1}{2} \norm{\PQ [\mb L +\mb S -\mb M]}{F}^2 + \lambda_L t_L + \lambda_S t_S \nonumber\\
&\mbox{s.t.}\quad &  \norm{\mb L}{*}\le t_L, \; \norm{\mb S}{1}\le t_S,
%&           &  \norm{\mb S}{1}\le t_S \nonumber.
\end{eqnarray}
obtained by introducing auxiliary variables $t_L$ and $t_S$. Now the objective function $g(\mb L, \mb S, t_L, t_S)$ is differentiable, with \begin{eqnarray}
&\nabla_{L} g(\mb L,\mb S,t_L,t_S) = \nabla_{S} g(\mb L, \mb S,t_L,t_S) =  \PQ[ \mb L + \mb S - \mb M ], \label{eqn:gradient_LS} \\
&\nabla_{t_L} g(\mb L,\mb S, t_L,t_S) =  \lambda_L, \quad
\nabla_{t_S} g(\mb L,\mb S,t_L,t_S)  = \lambda_S. \label{eqn:gradient_t}
\end{eqnarray}
A calculation, which we summarize in the following lemma, shows that the gradient  $\nabla g(\mb L, \mb S, t_L, t_S) = (\nabla_{L} g,\nabla_{S} g, \nabla_{t_L} g,\nabla_{t_S} g)$ is 2-Lipschitz:
\begin{lemma}\label{lem:lipschitz} For all $(\mb L,\mb S, t_L, t_S)$ and $(\mb L',\mb S', t_L', t_S')$ feasible to \eqref{eqn:penal_prob_epi},
\begin{equation}
\norm{\nabla g(\mb L, \mb S, t_L, t_S) - \nabla g(\mb L',\mb S', t_L', t_S')}{F} \le 2 \norm{ (\mb L,\mb S, t_L, t_S) - (\mb L',\mb S', t_L', t_S') }{F}.
\end{equation}
\end{lemma}

\begin{proof}
Based on \eqref{eqn:gradient_LS} and \eqref{eqn:gradient_t}, it follows directly that
\begin{flalign*}
 \norm{ \nabla g(\mb L,\mb S, t_L, t_S) - \nabla g(\mb L',\mb S', t_L', t_S') }{F}^2
%&=& 2 \norm{ \PQ[ \mb L + \mb S - \mb M ] - \PQ[ \mb L' + \mb S' - \mb M ] }{F}^2 + (\lambda_L - \lambda_L)^2 + (\lambda_S - \lambda_S)^2 \\
%&=& 2 \norm{ \PQ[ \mb L + \mb S - \mb M ] - \PQ[ \mb L' + \mb S' - \mb M ] }{F}^2 \\
%&=& 2 \norm{ \PQ[ \mb L + \mb S ] - \PQ[ \mb L' + \mb S' ] }{F}^2 \\
%&\le& 2 \norm{ \mb L + \mb S - \mb L' - \mb S' }{F}^2 \\
&\le 4 \norm{\mb L - \mb L'}{F}^2 + 4 \norm{\mb S - \mb S'}{F}^2 \\
&\le 4 \norm{ (\mb L,\mb S, t_L, t_S) - (\mb L',\mb S', t_L', t_S') }{F}^2,
\end{flalign*}
which implies the result.
\end{proof}

However, the Frank-Wolfe method still cannot deal with \eqref{eqn:penal_prob_epi}, since its feasible region is unbounded. If we could somehow obtain upper bounds on the optimal values of $t_L$ and $t_S$: $U_L \ge t_L^\star$ and $U_S \ge t_S^\star$, then we could solve the equivalent problem
\begin{eqnarray} \label{eqn:penal_prob_epi_bd}
&\min   \quad    &  \frac{1}{2} \norm{\PQ [\mb L +\mb S -\mb M]}{F}^2 + \lambda_L t_L + \lambda_S t_S \\
&\mbox{s.t.}\quad&  \norm{\mb L}{*}\le t_L\le U_L, \; \norm{\mb S}{1}\le t_S\le U_S \nonumber,
%&           &  \norm{\mb S}{1}\le t_S\le U_S \nonumber,
\end{eqnarray}
which now has a compact and convex feasible set. One simple way to obtain such $U_L$, $U_S$ is as follows. One trivial feasible solution to problem \eqref{eqn:penal_prob_epi} is $\mb L = \mb 0$, $\mb S = \mb 0$, $t_L = 0$, $t_S = 0$. This solution has objective value $\tfrac{1}{2} \norm{\PQ[\mb M]}{F}^2$. Hence, the optimal objective value is no larger than this. This implies that for any optimal $t_L^\star, t_S^\star$,
\begin{eqnarray}
t_L^\star \;\le\; \frac{1}{2\lambda_L} \norm{\PQ[ \mb M ]}{F}^2, \qquad t_S^\star \;\le\; \frac{1}{2\lambda_S} \norm{\PQ[\mb M]}{F}^2.
\end{eqnarray}
Hence, we can always choose
\begin{equation} \label{eqn:U1-U2}
U_L = \frac{1}{2\lambda_L} \norm{\PQ[ \mb M ]}{F}^2, \quad U_S = \frac{1}{2\lambda_S} \norm{\PQ[\mb M]}{F}^2
\end{equation}
to produce a valid, bounded feasible region. The following lemma bounds its diameter $D$:
\begin{lemma}\label{lem:diameter} The feasible set  $\mc D = \set{ (\mb L,\mb S, t_L, t_S) \mid \norm{\mb L}{*} \le t_L \le U_L , \; \norm{\mb S}{1} \le t_S \le U_S }$ has diameter $D \le \sqrt 5 \cdot \sqrt{U_L^2 + U_S^2 }$.
\end{lemma}

\begin{proof}
Since for any $\mb Z = (\mb L, \mb S, t_L, t_S)$, $\mb Z' = (\mb L', \mb S', t_L',t_S') \in \mc D$, we have
\begin{flalign*} \label{eqn:nc_diameter}
\norm{\mb Z - \mb Z'}{F}^2 &= \norm{\mb L-\mb L'}{F}^2+\norm{\mb S-\mb S'}{F}^2 + (t_L -t_L')^2 + (t_S - t_S')^2 \\
&\le (\norm{\mb L}{F}+\norm{\mb L'}{F})^2+(\norm{\mb S}{F}+\norm{\mb S'}{F})^2+ (t_L -t_L')^2 + (t_S - t_S')^2 \nonumber \\
&\le (\norm{\mb L}{*}+\norm{\mb L'}{*})^2+(\norm{\mb S}{1}+\norm{\mb S'}{1})^2 + (t_L -t_L')^2 + (t_S - t_S')^2 \\
&\le (U_L + U_L)^2 + (U_S + U_S)^2 + U_L^2 + U_S^2 \\
& = 5(U_L^2 + U_S^2),
\end{flalign*}
which implies the result.
\end{proof}

With these modifications, we can apply Frank-Wolfe directly to obtain a solution $(\widehat{\mb L},\widehat{\mb S},\widehat{t_L},\widehat{t_S})$ to \eqref{eqn:penal_prob_epi_bd}, and hence to produce a solution $(\widehat{\mb L},\widehat{\mb S})$ to the original problem \eqref{eqn:penalized_prob_main}. In subsection \ref{sec:FW-penalized}, we describe how to do this. Unfortunately, this straightforward solution has two main disadvantages. First, as in the norm constrained case, it produces only one-sparse updates to $\mb S$, which results in slow convergence. Second, the exact primal convergence rate in Theorem \ref{thm:fw_primal_conv} depends on the diameter of the feasible set, which in turn depends on the accuracy of our (crude) upper bounds $U_L$ and $U_S$. In subsection \ref{subsec:FW-T}, we show how to remedy both issues, yielding a Frank-Wolfe-Thresholding method that performs significantly better in practice.

\subsection{Frank-Wolfe for problem  \eqref{eqn:penal_prob_epi_bd}} \label{sec:FW-penalized}
Applying the Frank-Wolfe method in Algorithm \ref{alg:Frank-Wolfe} generates a sequence of iterates $\bm x^k = (\mb L^k, \mb S^k, t_L^k, t_S^k)$. Using the expressions for the gradient in \eqref{eqn:gradient_LS} and \eqref{eqn:gradient_t}, at each iteration, $\bm v^k = (\mb V_L^k, \mb V_S^k, V_{t_L}^k, V_{t_S}^k)$ is generated by solving the linearized subproblem
\begin{flalign}
\bm v^k \in \arg \min_{\bm v\in \mc D} \quad \innerprod{\PQ [\mb L^k + \mb S^k - \mb M]}{\mb V_L+ \mb V_S}+\lambda_L V_{t_L}+\lambda_S V_{t_S},
\end{flalign}
which can be decoupled into two independent subproblems,
\begin{flalign}
(\mb V_L^k, V_{t_L}^k)   &\in  \arg \hspace{-5mm} \min_{\norm{\mb V_L}{*} \le V_{t_L}\le U_L} g_L(\mb V_L, V_{t_L}) \doteq \innerprod{\PQ [\mb L^k + \mb S^k - \mb M]}{\mb V_L}+\lambda_L V_{t_L}     \label{eqn:V_L_V_t}\\
(\mb V_S^k, V_{t_S}^k)   &\in  \arg \hspace{-5mm} \min_{\norm{\mb V_S}{1} \le V_{t_S}\le U_S} g_S(\mb V_S, V_{t_S}) \doteq \innerprod{\PQ [\mb L^k + \mb S^k - \mb M]}{\mb V_S}+\lambda_S V_{t_S}.
\end{flalign}
Let us consider problem \eqref{eqn:V_L_V_t} first. Set
\begin{equation}
\mb D_L^k \in \arg \min_{\norm{\mb D_L}{*}\le 1} \;\; \hat g_L (\mb D_L)\doteq \innerprod{\PQ [\mb L^k + \mb S^k - \mb M]}{\mb D_L}+\lambda_L.
\end{equation}
Because $g_L(\mb V_L, V_{t_L})$ is a homogeneous function, i.e.,
$g_L(\alpha \mb V_L, \alpha V_{t_L}) = \alpha g_L(\mb V_L, V_{t_L})$, for any $\alpha \in \reals$, its optimal value $g(\mb V_L^k, V_{t_L}^k)= V_{t_L}^k \hat g_L(\mb D_L^k)$. Hence $V_{t_L}^k = U_L$ if $\hat g_L(\mb D_L^k)<0$, and $V_{t_L}^k = 0$ if $\hat g_L(\mb D_L^k)>0$. From this observation, it can be easily verified (see also \cite[Lemma 1]{harchaoui2013conditional} for a more general result) that
\begin{equation} \label{eqn:pen_linear_sub}
(\mb V^k_L, V^k_{t_L}) \in \begin{cases} \set{(\mb 0, 0)} &\mbox{if }  \hat g_L (\mb D_L^k) > 0 \\
\mbox{conv}\{(\mb 0, 0), U_L(\mb D_L^k, 1) \}& \mbox{if } \hat g_L (\mb D_L^k) = 0 \\
\set{U_L(\mb D_L^k, 1)} & \mbox{if } \hat g_L (\mb D_L^k) < 0.   \end{cases}
\end{equation}
In a similar manner, we can update $(\mb V_S^k, V_{t_S}^k)$. This leads fairly directly to the implementation of the Frank-Wolfe method for problem \eqref{eqn:penal_prob_epi_bd}, described in
Algorithm \ref{alg:Frank_Wolfe_penal}. As a direct corollary of Theorem \ref{thm:fw_primal_conv}, using parameters calculated in Lemmas \ref{lem:lipschitz} and \ref{lem:diameter}, we have
\begin{corollary}\label{cor:FW_penal}
Let $\bm x^\star = (\mb L^\star, \mb S^\star, t_L^\star, t_S^\star)$ be an optimal solution to \eqref{eqn:penal_prob_epi_bd}. For $\{\bm x^k\}$ generated by Algorithm $\ref{alg:Frank_Wolfe_penal}$, we have %\footnote{A more careful calculation would lead us to $g(\bm x^{k})-g(\bm x^\star) \le \frac{16(U_L^2+U_S^2)}{k+2}$, which we also include in the appendix.}
for $k = 0,\; 1, \;2, \;\ldots,$
\begin{equation}
g(\bm x^{k})-g(\bm x^\star) \le \frac{20 ( U_L^2+U_S^2 )}{k+2}.
\end{equation}
\end{corollary}

\begin{proof}
Applying Theorem \ref{thm:fw_primal_conv} with parameters calculated in Lemmas \ref{lem:lipschitz} and \ref{lem:diameter}, we directly have
\begin{equation}\label{eqn:cor_20}
g(\bm x^k) - g(\bm x^\star) \le \frac{2\cdot 2 \cdot \left(\sqrt{5(U_L^2 +U_S^2)}\right)^2}{k+2} = \frac{20(U_L^2 +U_S^2)}{k+2}.
\end{equation}
A more careful calculation below slightly improves the constant in \eqref{eqn:cor_20}.
\begin{eqnarray}
g(\bm x^{k+1}) & = & g(\bm x^k + \gamma (\bm v^k - \bm x^k)) \nonumber\\
&\le& g(\bm x^k) + \gamma \innerprod{\nabla g(\bm x^k)}{\bm v^k-\bm x^k}+ \gamma^2 \norm{\mb V_L^k - \mb L^k}{F}^2 +\gamma^2 \norm{\mb V_S^k - \mb S^k}{F}^2 \nonumber\\
&\le& g(\bm x^k) + \gamma \innerprod{\nabla g(\bm x^k)}{\bm v^k-\bm x^k}+ 4\gamma^2(U_L^2 + U_S^2), \label{eqn:cor_careful}
\end{eqnarray}
where the second line holds by noting that $g$ is only linear in $t_L$ and $t_S$; the last line holds as
\begin{eqnarray*}
\norm{\mb V_L^k - \mb L^k}{F}^2 &\le& (\norm{\mb V^k_L}{F}+ \norm{\mb L^k}{F})^2 \le (U_L+U_L)^2 = 4U_L^2, \quad\mbox{and} \\
\norm{\mb V_S^k - \mb S^k}{F}^2 &\le& (\norm{\mb V^k_S}{F}+ \norm{\mb S^k}{F})^2 \le (U_S+U_S)^2 = 4U_S^2.
\end{eqnarray*}
Following the arguments in the proof of Theorem 1 with \eqref{eqn:fw_pf_prim_conv_first} replaced by \eqref{eqn:cor_careful}, we can easily obtain that
$$
g(\bm x^k) - g(\bm x^\star) \le  \frac{16(U_L^2 +U_S^2)}{k+2}.
$$
\end{proof}

\begin{algorithm}[tb]
   \caption{Frank-Wolfe method for problem \eqref{eqn:penal_prob_epi_bd}}
   \label{alg:Frank_Wolfe_penal}
\begin{algorithmic}[1]
   \STATE {\bfseries Initialization:} $\mb L^0 = \mb S^0 = \mb 0;$ $t_L^0 = t_S^0 = 0$;
   \FOR{$k=0,\; 1,\; 2, \; \ldots$}
   \STATE $\mb D_L^k \in \arg \min_{\norm{\mb D_L}{*}\le 1} \< \PQ[\mb L^k + \mb S^k - \mb M], \; \mb D_L\>$; \label{line:D_L}
   \STATE $\mb D_S^k \in \arg \min_{\norm{\mb D_S}{1}\le 1} \< \PQ[\mb L^k + \mb S^k - \mb M], \; \mb D_S\>$;  \label{line:D_S}
   \IF{ $\lambda_L \ge  -\< \PQ[\mb L^k + \mb S^k - \mb M], \; \mb D_L^k\>$}
   \STATE $\mb V_L^k = \mb 0$; $V^k_{t_L} = 0$
   \ELSE
   \STATE $\mb V_L^k = U_L \mb D_L^k$, $V^k_{t_L} = U_L$;
   \ENDIF

   \IF{ $\lambda_S \ge  -\< \PQ[\mb L^k + \mb S^k - \mb M], \; \mb D_S^k\>$}
   \STATE $\mb V_S^k = \mb 0$; $V^k_{t_S} = 0$;
   \ELSE
   \STATE $\mb V_S^k = U_S \mb D_S^k$, $V^k_{t_S} = U_S$;
   \ENDIF

   \STATE $\gamma = \frac{2}{k+2}$;
   \STATE $\mb L^{k+1} = (1-\gamma) \mb L^k + \gamma \mb V^k_L$, $t_L^{k+1}=(1-\gamma)t_L^k+\gamma V^k_{t_L}$;
   \STATE $\mb S^{k+1} = (1-\gamma) \mb S^k + \gamma \mb V^k_S$, $t_S^{k+1}=(1-\gamma)t_S^k+\gamma V^k_{t_S}$;

   \ENDFOR
\end{algorithmic}
\end{algorithm}

In addition to the above convergence result, another major advantage of Algorithm \ref{alg:Frank_Wolfe_penal} is the simplicity of the update rules (lines \ref{line:D_L}-\ref{line:D_S} in Algorithm \ref{alg:Frank_Wolfe_penal}). Both have closed-form solutions that can be computed in time (essentially) linearly dependent on the size of the input.

However, two clear limitations substantially hinder Algorithm \ref{alg:Frank_Wolfe_penal}'s efficiency. First, as in the norm constrained case, $\mb V_S^k$ has only one nonzero entry, so $\mb S$ has a one-sparse update in each iteration. Second, the exact rate of convergence relies on our (crude) guesses of $U_L$ and $U_S$ (Corollary \ref{cor:FW_penal}).
%The red curves in Figure \ref{fig:synthetic} show these effects on the practical convergence of the algorithm, on a simulated example of size $1,000 \times 1,000$, in which about $1\%$ of the entries in the target sparse matrix $\mb S_0$ are nonzero.\footnote{In specific, the data are generated in Matlab as $m = 1000; \; n = 1000; \; r = 5; \; L_0 = randn(m,r) * randn(r,n); \; Omega = ones(m,n); \; S_0 = 100*randn(m,n).*(rand(m,n) < 0.01); \; M = L_0 + S_0 + 0.1*randn(m,n);$} As we can observe, progress is very slow.
In the next subsection, we present remedies to resolve both issues.

\subsection{FW-T algorithm: combining Frank-Wolfe and proximal methods} \label{subsec:FW-T}

To alleviate the difficulties faced by Algorithm \ref{alg:Frank_Wolfe_penal}, we propose a new algorithm called Frank-Wolfe-Thresholding (FW-T) (Algorithm \ref{alg:FW-T}), that combines a modified FW step with a proximal gradient step.
%In Figure \ref{fig:synthetic}, we compare Algorithms \ref{alg:Frank_Wolfe_penal} and \ref{alg:FW-T} on synthetic data.
%In this example, the FW-T method is clearly more efficient in recovering $\mb L_0$ and $\mb S_0$.
Below we highlight the  key features of FW-T.

\paragraph{\textbf{Proximal gradient step for $\mb S$}} To update $\mb S$ in a more efficient way, we incorporate an additional proximal gradient step for $\mb S$. At iteration $k$, let $(\mb L^{k+\frac{1}{2}}, \mb S^{k+\frac{1}{2}})$ be the result produced by Frank-Wolfe step. To produce the next iterate, we retain the low-rank term $\mb L^{k+\frac{1}{2}}$, but execute a proximal gradient step for the function $f(\mb L^{k+\frac{1}{2}}, \mb S)$ at the point $\mb S^{k+\frac{1}{2}}$, i.e.
\begin{flalign}\label{eqn:threshold_exp}
\mb S^{k+1} \in  &\arg \min_{\mb S} \innerprod{\nabla_{\mb S} f(\mb L^{k+\frac{1}{2}}, \mb S^{k+\frac{1}{2}})}{\;\mb S - \mb S^{k+\frac{1}{2}}} + \frac{1}{2}\norm{\mb S - \mb S^{k+\frac{1}{2}}}{F}^2 + \lambda_S
\norm{\mb S}{1} \nonumber \\
=& \arg \min_{\mb S} \innerprod{\PQ [\mb L^{k+\frac{1}{2}}+ \mb S^{k+\frac{1}{2}} - \mb M]}{\mb S - \mb S^{k+\frac{1}{2}}} + \frac{1}{2}\norm{\mb S - \mb S^{k+\frac{1}{2}}}{F}^2 + \lambda_S
\norm{\mb S}{1}
\end{flalign}
which can be easily computed using the soft-thresholding operator:
\begin{equation}
\mb S^{k+1} = \mc T_{\lambda_S} \left[ \mb S^{k+\frac{1}{2}} - \PQ [ \mb L^{k+\frac{1}{2}} + \mb S^{k+\frac{1}{2}} - \mb M ] \right].
\end{equation}
\paragraph{\textbf{Exact line search}} For the Frank-Wolfe step, instead of choosing the fixed step length $\frac{2}{k+2}$, we implement an exact line search  by solving a two-dimensional quadratic problem \eqref{eqn:QP_step}, as in \cite{harchaoui2013conditional}. This modification turns out to be crucial to achieve a primal convergence result that only weakly depends on  the tightness of our guesses $U_L$ and $U_S$.
%does not depend on the tightness of our guesses $U_L$ and $U_S$. However, the convergence of the surrogate duality gap
%\begin{equation}
%d(\bm x^k) = \innerprod{\bm x^k - \bm v^k}{\nabla g(\bm x^k)}
%\end{equation}
%still depends on $U_L$ and $U_S$. Since this quantity is used as part of our stopping criterion, we still prefer a tight choice of upper bounds $U_L$ and $U_S$.
\paragraph{\textbf{Adaptive updates of $U_L$ and $U_S$}}
%Since the correct values are difficult to estimate ahead of time, we introduce a scheme which dynamically updates these quantities based on the intermediate outputs of the algorithm.
We initialize $U_L$ and $U_S$ using the crude bound \eqref{eqn:U1-U2}. Then, at the end of the $k$-iteration, we respectively update
\begin{equation}
    U_L^{k+1} = g(\mb L^{k+1}, \mb S^{k+1}, t_L^{k+1}, t_S^{k+1})/\lambda_L, \quad U_S^{k+1} = g(\mb L^{k+1}, \mb S^{k+1}, t_L^{k+1}, t_S^{k+1})/\lambda_S.
\end{equation}
This scheme maintains the property that $U_L^{k+1} \ge  t_L^\star$ and $U_S^{k+1} \ge t_S^\star$. Moreover, we prove (Lemma \ref{lem:non_dec}) that $g$ is non-increasing through our algorithm, and so this scheme produces a sequence of tighter upper bounds for $U_L^\star$ and $U_S^\star$. Although this dynamic scheme does not improve the theoretical convergence result, some acceleration is empirically exhibited.

\begin{algorithm}[t!]
   \caption{FW-T method for problem \eqref{eqn:penalized_prob_main}}
   \label{alg:FW-T}
\begin{algorithmic}[1]
   \STATE {\bfseries Input:} data matrix $\mb M \in \reals^{m\times n}$; weights $\lambda_L$, $\lambda_S > 0$;
                             max iteration number $T$;
   \STATE {\bfseries Initialization:} $\mb L^0 = \mb S^0 = \mb 0;$ $t_L^0 = t_S^0 = 0$; $U_L^0 = g(\mb L^0, \mb S^0, t_L^0, t_S^0)/\lambda_L$; $U_S^0 = g(\mb L^0, \mb S^0, t_L^0, t_S^0)/\lambda_S$;
   \FOR{$k=0,\; 1,\; 2, \; \cdots, \; T$}
   \STATE {\em same as lines 3-14 in Algorithm \ref{alg:Frank_Wolfe_penal};}
   \STATE $\bigg( \mb L^{k+\frac{1}{2}}, \mb S^{k+\frac{1}{2}}, t_L^{k+\frac{1}{2}}, t_S^{K+\frac{1}{2}} \bigg)$ is computed as an optimizer to\vspace{-.15mm}
    \begin{eqnarray}\label{eqn:QP_step}
      &\min\quad& \frac{1}{2} \norm{\PQ [\mb L +\mb S -\mb M]}{F}^2 + \lambda_L t_L + \lambda_S t_S \\
      &\mbox{s.t.}\quad& \left(
                      \begin{array}{c}
                        \mb L \\
                        t_L \\
                      \end{array}
                    \right)
      \in \mbox{conv}\set{\left(
                      \begin{array}{c}
                        \mb L^k \\
                        t_L^k \\
                      \end{array}
                    \right), \left(
                      \begin{array}{c}
                        \mb V^k_L \\
                        V^k_{t_L} \\
                      \end{array}
                    \right) }
                    \nonumber\\
                    &\quad&
                      \left(
                      \begin{array}{c}
                        \mb S \\
                        t_S \\
                      \end{array}
                    \right)
      \in \mbox{conv}\set{\left(
                      \begin{array}{c}
                        \mb S^k \\
                        t_S^k \\
                      \end{array}
                    \right), \left(
                      \begin{array}{c}
                        \mb V^k_S \\
                        V^k_{t_S} \\
                      \end{array}
                    \right) }; \nonumber
    \end{eqnarray}
   \STATE \label{eqn:threshold} $\mb S^{k+1} = \mc T\big[\mb S^{k+\frac{1}{2}} - \PQ [\mb L^{k+\frac{1}{2}}+\mb S^{k+\frac{1}{2}}-\mb M], \lambda_S \big];$
   \STATE $\mb L^{k+1} = \mb L^{k+\frac{1}{2}}$, $t_L^{k+1} = t_L^{k+\frac{1}{2}}$; $t_S^{k+1 } = \norm{\mb S^{k+1}}{1}$;
   %\STATE $t_L^{k+1} = \norm{\mb L^{k+1}}{*}$;
   \STATE \label{eqn:update_U_L}$U_L^{k+1} = g(\mb L^{k+1}, \mb S^{k+1}, t_L^{k+1}, t_S^{k+1})/\lambda_L;$
   \STATE \label{eqn:update_U_S}$U_S^{k+1} = g(\mb L^{k+1}, \mb S^{k+1}, t_L^{k+1}, t_S^{k+1})/\lambda_S;$
   \ENDFOR
\end{algorithmic}
\end{algorithm}

\paragraph{\textbf{Convergence analysis}} Since both the FW step and the proximal gradient step do not increase the objective value, we can easily recognize FW-T method as a descent algorithm:
\begin{lemma} \label{lem:non_dec} Let $\{(\mb L^k,\mb S^k, t_L^k, t_S^k)\}$ be the sequence of iterates produced by the FW-T algorithm. For each $k = 0, 1, 2 \cdots$,
\begin{equation}
g(\mb L^{k+1}, \mb S^{k+1}, t_L^{k+1}, t_S^{k+1})
\le g(\mb L^{k+\frac{1}{2}}, \mb S^{k+\frac{1}{2}}, t_L^{k+\frac{1}{2}}, t_S^{k+\frac{1}{2}})
\le g(\mb L^k, \mb S^k, t_L^k, t_S^k).
\end{equation}
\end{lemma}

\begin{proof}
Since $(\mb L^k, \mb S^k, t_L^k, t_S^k)$ is always feasible to the quadratic program \eqref{eqn:QP_step},
\begin{equation} \label{eqn:k_to_k_half}
g(\mb L^{k+\frac{1}{2}}, \mb S^{k+\frac{1}{2}}, t_L^{k+\frac{1}{2}}, t_S^{k+\frac{1}{2}})
\le g(\mb L^k, \mb S^k, t_L^k, t_S^k).
\end{equation}
Based on \eqref{eqn:threshold_exp}, the threshold step (line 6 in Algorithm 3) can be written as
\begin{flalign*}
\mb S^{k+1} =& \arg \min_{\mb S} \quad \hat g^{k+\frac{1}{2}}(\mb S)\doteq \frac{1}{2} \norm{\PQ [\mb L^{k+\frac{1}{2}} + \mb S^{k+\frac{1}{2}} - \mb M]}{F}^2 + \lambda_L t_L^{k+\frac{1}{2}} + \lambda_S \norm{\mb S}{1} \\
&\qquad \qquad \qquad +\<\PQ[\mb L^{k+\frac{1}{2}} + \mb S^{k+\frac{1}{2}} - \mb M], \; \mb S - \mb S^{k+\frac{1}{2}} \>
+ \frac{1}{2} \norm{\mb S - \mb S^{k+\frac{1}{2}}}{F}^2.
\end{flalign*}
The following properties of $\hat g^{k+\frac{1}{2}}(\cdot)$ can be easily verified
$$
\hat{g}^{k+\frac{1}{2}}(\mb S^{k+\frac{1}{2}}) = g(\mb L^{k+\frac{1}{2}},\mb S^{k+\frac{1}{2}}, t_L^{k+\frac{1}{2}}, \|\mb S^{k+\frac{1}{2}}\|_1)\le g(\mb L^{k+\frac{1}{2}},\mb S^{k+\frac{1}{2}}, t_L^{k+\frac{1}{2}}, t_S^{k+\frac{1}{2}});
$$
$$\emph{}
\hat{g}^{k+\frac{1}{2}}(\mb S) \ge g(\mb L^{k+\frac{1}{2}}, \mb S, t_L^{k+\frac{1}{2}}, \norm{\mb S}{1}), \quad \mbox{for any } \mb S.
$$
Therefore, we have
\begin{flalign}\label{eqn:k_half_to_k}
g(\mb L^{k+1},\mb S^{k+1}, t_L^{k+1}, t_S^{k+1}) &= g(\mb L^{k+\frac{1}{2}},\mb S^{k+1}, t_L^{k+\frac{1}{2}}, t_S^{k+1}) \le  \hat{g}^{k+\frac{1}{2}}(\mb S^{k+1}) \nonumber\\
& \le \hat{g}^{k+\frac{1}{2}}(\mb S^{k+\frac{1}{2}}) \le g(\mb L^{k+\frac{1}{2}},\mb S^{k+\frac{1}{2}}, t_L^{k+\frac{1}{2}}, t_S^{k+\frac{1}{2}})
\end{flalign}
Combining \eqref{eqn:k_to_k_half} and \eqref{eqn:k_half_to_k}, we obtain
$$
g(\mb L^{k+1}, \mb S^{k+1}, t_L^{k+1}, t_S^{k+1})
\le g(\mb L^{k+\frac{1}{2}}, \mb S^{k+\frac{1}{2}}, t_L^{k+\frac{1}{2}}, t_S^{k+\frac{1}{2}})
\le g(\mb L^k, \mb S^k, t_L^k, t_S^k).
$$
\end{proof}

Moreover, we can establish primal convergence (almost) independent of $U^0_L$ and $U^0_S$:
\begin{theorem} \label{thm:pen_main_thm}
Let $r_L^\star$ and $r_S^\star$ be the smallest radii such that
\begin{equation}
\set{(\mb L, \mb S) \;\middle|\; f(\mb L, \mb S) \le g(\mb L^0, \mb S^0, t_L^0, t_S^0) = \frac{1}{2}\norm{\PQ[\mb M]}{F}^2}\subseteq {\overline{B( r_L^\star)}} \times {\overline{B(r_S^\star)}},
\end{equation}
where $\overline{B(r)} \doteq \set{\mb X \in \reals^{m\times n} \vert \norm{\mb X}{F} \le r}$
for any $r \ge 0$.\footnote{Since the objective function in problem \eqref{eqn:penalized_prob_main} is coercive, i.e. $\lim_{k\to +\infty}f(\mb L^k, \mb S^k) = +\infty$
for any sequence $(\mb L^k, \mb S^k)$ such that $\lim_{k\to +\infty}\norm{(\mb L^k, \mb S^k)}{F} = +\infty$, clearly
$r_L^\star \ge 0$ and $r_S^\star \ge 0$ exist.} Then for the sequence $\{(\mb L^k, \mb S^k, t_L^k, t_S^k)\}$ generated by Algorithm \ref{alg:FW-T}, we have
\begin{flalign}
&g(\mb L^k, \mb S^k, t_L^k, t_S^k) - g(\mb L^\star, \mb S^\star, t_L^\star, t_S^\star) \\
&\qquad \qquad \qquad \qquad\le \frac{\min \{ 4(t_L^\star+r_L^\star)^2+4(t_S^\star+r_S^\star)^2, \; 16(U_L^0)^2+16(U_S^0)^2\}}{k+2}. \nonumber
\end{flalign}
\end{theorem}
\begin{proof}
For notational convenience, we denote
$$
\bm x^k = (\mb L^k, \mb S^k, t_L^k, t_S^k), \; \bm x^\star = (\mb L^\star, \mb S^\star, t_L^\star, t_S^\star) \; \mbox{and} \; \bm v^k = (\mb V_L^k, \mb V_S^k, \bm V_{t_L}^k,  \bm V_{t_S}^k).
$$
For any point $\bm x = (\mb L, \mb S, t_L, t_S)\in \reals^{m\times n}\times \reals^{m \times n} \times \reals \times \reals$, we adopt the notation that $\mb L[\bm x] = \mb L$, $\mb S[\bm x] = \mb S$, $t_L[\bm x] = t_L$ and $t_S[\bm x] = t_S$.

Since $g(\bm x^k) - g(\bm x^\star) \le \frac{16(U_L^0)^2+16(U_S^0)^2}{k+2}$ can be easily established following the proof of Corollary \ref{cor:FW_penal}, below we will focus on the other part that $g(\bm x^k) - g(\bm x^\star) \le \frac{4(t_L^\star+r_L^\star)^2+4(t_S^\star+r_S^\star)^2}{k+2}$.

Let us first make two simple observations.

Since $f(\mb L^\star, \mb S^\star) \le g (\mb L^k, \mb S^k, t_L^k, t_S^k)$,
we have
\begin{equation}
U_L^{k} = g(\mb L^{k}, \mb S^{k}, t_L^k, t_S^{k})/\lambda_L \ge  t_L^\star \quad \mbox{and} \quad U_S^{k} = g(\mb L^{k}, \mb S^{k}, t_L^{k}, t_S^{k})/\lambda_S
\ge  t_S^\star.
\end{equation}
Therefore, our $U_L^k$ and $U_S^k$ always bound $t_L^\star$ and $t_S^\star$ from above.

From Lemma \ref{lem:non_dec}, $g(\mb L^k, \mb S^k, t_L^k, t_S^k)$ is non-increasing,
$$
f(\mb L^k, \mb S^k)\le g(\mb L^k, \mb S^k, t_L^k, t_S^k) \le g(\mb L^0, \mb S^0, t_L^0, t_S^0),
$$
which implies that $(\mb L^k, \mb S^k)\subseteq \overline{B(r_L^\star)} \times \overline{B(r_S^\star)}$, i.e.
$\norm{\mb L^k}{F}\le r_L^\star$ and $\norm{\mb S^k}{F}\le r_S^\star.$

Let us now consider the $k$-th iteration. Similar to the proof in \cite{harchaoui2013conditional}, we
introduce the auxiliary point $\bm v^k_+ = ( \frac{t_L^\star}{U_L^k} \mb V^k_L, \frac{t_S^\star}{U_S^k} \mb V^k_S, \frac{t_L^\star}{U_L^k} \mb V^k_{t_L}, \frac{t_S^\star}{U_S^k} \mb V^k_{t_S})$.
Then based on our argument for \eqref{eqn:pen_linear_sub}, it can be easily verified that
\begin{eqnarray}
(\mb L[\bm v^k_+], t_L[\bm v^k_+])  &\in& \arg \min_{\norm{\mb V_L}{*} \le V_{t_L}\le t_L^\star} g_L(\mb V_L, V_{t_L}) \label{eqn:aux_linear_sub_1}\\
(\mb S[\bm v^k_+], t_S[\bm v^k_+])  &\in& \arg \min_{\norm{\mb V_S}{1} \le V_{t_S}\le t_S^\star} g_S(\mb V_S, V_{t_S}) \label{eqn:aux_linear_sub_2}.
\end{eqnarray}
Recall $\gamma = \frac{2}{k+2}$. We have
\begin{eqnarray*}
  &&g(\bm x^{k+\frac{1}{2}}) \\
  &\le& g(\bm x^k + \gamma (\bm v^k_+ - \bm x^k)) \\
  &\le& g(\bm x^k)+\gamma \< \nabla g(\bm x_k),\; \bm v^k_+ - \bm x^k\>         + \gamma^2\left( \norm{\mb L[\bm v^k_+] - \mb L[\bm x^k]}{F}^2 + \norm{\mb S[\bm v^k_+] - \mb S[\bm x^k]}{F}^2 \right)\\
  &\le& g(\bm x^k)+ \gamma \left(g_L(\mb L[\bm v^k_+ - \bm x^k], t_L [\bm v^k_+ - \bm x^k]) + g_S(\mb S[\bm v^k_+ - \bm x^k], t_S[\bm v^k_+ - \bm x^k]) \right) \\
   &&  \hspace{70mm}      + \gamma^2  \left((t_L^\star+  r_L^\star)^2+ (t_S^\star+  r_S^\star)^2\right)\\
     &\le& g(\bm x^k)+ \gamma \left(g_L(\mb L[\bm x^\star - \bm x^k], t_L [\bm x^\star - \bm x^k]) + g_S(\mb S[\bm  x^\star- \bm x^k], t_S[\bm  x^\star - \bm x^k]) \right) \\
   &&  \hspace{70mm}      + \gamma^2  \left((t_L^\star+  r_L^\star)^2+ (t_S^\star+  r_S^\star)^2\right)\\
   &=& g(\bm x^k)+\gamma \< \nabla g(\bm x^k),\; \bm  x^\star- \bm x^k\>
         + \gamma^2  \left((t_L^\star+  r_L^\star)^2+ (t_S^\star+  r_S^\star)^2\right)\\
  &\le& g(\bm x^k)+\gamma \left(g(\bm x^\star)-g(\bm x^k)\right)
         + \gamma^2  \left((t_L^\star+  r_L^\star)^2+ (t_S^\star+  r_S^\star)^2\right),\\
\end{eqnarray*}
where the first inequality holds since $\bm x^k + \gamma (\bm v^k_+ - \bm x^k)$ is feasible
to the quadratic program \eqref{eqn:QP_step} while $\bm x^{k+\frac{1}{2}}$ minimizes it;
the third inequality is due to the facts that
\begin{eqnarray*}
  \norm{\mb L[\bm v^k_+] - \mb L[\bm x^k]}{F} &\le& \norm{\mb L[\bm v^k_+]}{F} + \norm{\mb L[\bm x^k]}{F}  \le \norm{\mb L[\bm v^k_+]}{*} + \norm{\mb L[\bm x^k]}{F} \le t_L^\star+  r_L^\star \\
  \norm{\mb S[\bm v^k_+] - \mb S[\bm x^k]}{F} &\le& \norm{\mb S[\bm v^k_+]}{F} + \norm{\mb S[\bm x^k]}{F} \le \norm{\mb S[\bm v^k_+]}{1} + \norm{\mb S[\bm x^k]}{F}\le t_S^\star+  r_S^\star;
\end{eqnarray*}
the fourth inequality holds as $(\mb L[\bm x^\star], t_L[\bm x^\star])$ and $(\mb S[\bm x^\star], t_S[\bm x^\star])$
are respectively feasible to \eqref{eqn:aux_linear_sub_1} and \eqref{eqn:aux_linear_sub_2}
while $(\mb L[\bm v^k_+], t_L[\bm v^k_+])$ and $(\mb S[\bm v^k_+], t_S[\bm v^k_+])$ respectively
minimize \eqref{eqn:aux_linear_sub_1} and \eqref{eqn:aux_linear_sub_2};

Therefore, we obtain
$$
g(\bm x^{k+\frac{1}{2}})  - g(\bm x^\star) \le (1-\gamma) \left(g(\bm x^k)  - g(\bm x^\star) \right)
+ \gamma^2 \left((t_L^\star+  r_L^\star)^2+ (t_S^\star+  r_S^\star)^2\right).
$$
Moreover, by Lemma \ref{lem:non_dec}, we have
$$
g(\bm x^{k+1}) \le g(\bm x^{k+\frac{1}{2}}).
$$
Thus, we obtain the recurrence
$$
g(\bm x^{k+1})  - g(\bm x^\star) \le (1-\gamma) \left(g(\bm x^k)  - g(\bm x^\star) \right)
+ \gamma^2 \left((t_L^\star+  r_L^\star)^2+ (t_S^\star+  r_S^\star)^2\right).
$$
Applying mathematical induction, one can easily obtain that
$$
g(\mb L^k, \mb S^k, t_L^k, t_S^k) - g(\mb L^\star, \mb S^\star, t_L^\star, t_S^\star) \le \frac{4\left((t_L^\star+r_L^\star)^2+(t_S^\star+r_S^\star)^2\right)}{k+2}.
$$
\end{proof}

Since $U_L^0$ and $U_S^0$ are quite crude upper bounds for $t_L^\star$ and $t_S^\star$, $16(U_L^0)^2+16(U_S^0)^2$ could be much larger than $4(t_L^\star+r_L^\star)^2+4(t_S^\star+r_S^\star)^2$. Therefore, this primal convergence results depend on $U_L^0$ and $U_S^0$ in a very weak manner.
%Note that the constant on the rate of primal convergence only depends on several intrinsic parameters, namely $U_L^\star$, $U_S^\star$,  $R_1^\star$ and $R_1^\star$.

However, the convergence result of the surrogate duality gap $d(\bm x^k)$ still hinges upon the upper bounds:
\begin{theorem}\label{thm:pen_dual_thm}
 Let $\bm x^k$ denote  $(\mb L^k,\mb S^k, t_L^k, t_S^k)$ generated by Algorithm \ref{alg:FW-T}. Then for any $K\ge 1$, there exists $1\le \tilde k \le K$ such that
\begin{equation}
g(\bm x^{\tilde k}) -  g(\bm x^\star) \le
d(\bm x^{\tilde k}) \le \frac{48 \left( (U_L^0)^2 + (U_S^0)^2  \right)}{K+2}.
\end{equation}
\end{theorem}
\begin{proof}
Define $\Delta^k = g(\bm x^k) - g(\bm x^\star)$. Following \eqref{eqn:cor_careful}, we have
\begin{equation}
\Delta^{k+1} \le \Delta^k + \gamma \innerprod{\nabla g(\bm x^k)}{\bm v^k-\bm x^k}+ 4\gamma^2\left((U_L^0)^2 + (U_S^0)^2\right). \label{eqn:thm5_dual}
\end{equation}
Then following the arguments in the proof of Theorem 2 with \eqref{eqn:fw_pf_dual_conv_1} replaced by \eqref{eqn:thm5_dual}, we can easily obtain the result.
\end{proof}

\paragraph{\textbf{Stopping criterion}} Compared to the convergence of $g(\bm x^k)$ (Theorem \ref{thm:pen_main_thm}), the convergence result for $d(\bm x^k)$ can be much slower (Theorem \ref{thm:pen_dual_thm}). Therefore, here the surrogate duality gap $d(\cdot)$ is not that suitable to serve as a stopping criterion. Consequently, in our implementation, we terminate Algorithm \ref{alg:FW-T} if
\begin{equation}\label{eqn:stopping}
%\left(g(\bm x^{k+1})- g(\bm x^{k})\right)/{\norm{\PQ[\mb M_0]}{F}^2} \le \eps,
\left|g(\bm x^{k+1})- g(\bm x^{k})\right|/g(\bm x^{k}) \le \eps,
\end{equation}
for five consecutive iterations.

\section{Numerical Experiments} \label{sec:num}
In this section, we report numerical results obtained by applying our FW-T method (Algorithm \ref{alg:FW-T}) to problem \eqref{eqn:SPCP} with real data arising from applications considered in \cite{candes2011robust}:
{\em foreground/background separation in surveillance videos}, and {\em shadow and specularity removal from face images}.

Given observations $\set{\mb M_0(i,j) \;\middle|\; (i,j)\in \Omega}$, where $\Omega \subseteq \set{1,\ldots, m}\times \set{1,\ldots, n}$ is the index set of the observable entries in $\mb M_0 \in \reals^{m\times n}$, we assigned weights $$\lambda_L = \delta \rho \norm{\PO[\mb M_0]}{F} \quad \mbox{and} \quad \lambda_S = \delta \sqrt \rho \norm{\PO[\mb M_0]}{F}/\sqrt{\max(m,n)}$$ to problem \eqref{eqn:SPCP}, \footnote{The ratio $\lambda_L/\lambda_S = \sqrt{\rho\max(m,n)}$ follows the suggestion in \cite{candes2011robust}. For applications in computer vision at least, our choices in $\lambda_L$ and $\lambda_S$ seem to be quite robust, although it is possible to improve the performance by making slight adjustments to our current settings of $\lambda_L$ and $\lambda_S$.} where $\rho = |\Omega|/{mn}$ and $\delta $ is chosen as $0.001$ for the surveillance problem and $0.01$ for the face problem.

We compared our FW-T method with the popular first-order methods {\em iterative soft-thresholding algorithm} (ISTA) and {\em fast iterative soft-thresholding algorithm} (FISTA) \cite{beck2009fast}, both of whose implementations used partial {\em singular value decomposition} (SVD). In subsection \ref{app:alg}, we provided detailed descriptions and implementations of ISTA and FISTA.

We set $\eps = 10^{-3}$ in FW-T's stopping criterion \eqref{eqn:stopping},\footnote{As discussed in \cite{yang2011alternating, yang2013linearized}, with noisy data, solving optimization problems to high accuracy does not necessarily improve the recovery quality. Consequently, we set $\eps$ to a modest value.} and terminated ISTA and FISTA whenever they reached the objective value returned by the FW-T method.\footnote{All codes are available at: https://sites.google.com/site/mucun1988/publi} All the experiments were conducted on a computer with Intel Xeon E5-2630 Processor (12 cores at 2.4 GHz), and 64GB RAM running MATLAB R2012b (64 bits).

\subsection{ISTA \& FISTA for problem \eqref{eqn:SPCP}}\label{app:alg}
{\em Iterative soft-thresholding algorithm} (ISTA), is an efficient way to tackle unconstrained nonsmooth optimization problem especially at large scale. ISTA follows the general idea by iteratively minimizing an upper bound of the original objective. In particular, when applied to problem  \eqref{eqn:SPCP} of our interest, ISTA updates $(\mb L, \mb S)$ for the $k$-th iteration by solving
\begin{flalign} \label{eqn:ista}
(\mb L^{k+1}, \mb S^{k+1}) = \arg \min_{\mb L, \mb S} \;
&\innerprod{\left( \begin{array}{ll} \nabla_{\mb L} l(\mb L^k, \mb S^k)\\\nabla_{\mb S} l(\mb L^k, \mb S^k)   \end{array} \right)}
{\left( \begin{array}{lll} \mb L - \mb L^k  \\ \mb S - \mb S^k  \end{array} \right)}
+ \\
& \qquad  \qquad \quad \frac{L_f}{2}\norm{\left( \begin{array}{lll} \mb L \\ \mb S \end{array} \right)-
\left( \begin{array}{lll} \mb L^k \\ \mb S^k \end{array} \right)}{F}^2
+\lambda_L \norm{\mb L}{*} + \lambda_S \norm{\mb S}{1}.  \nonumber
\end{flalign}
Here $L_f = 2$ denotes the Lipschitz constant of $\nabla l(\mb L, \mb S)$ with respect to $(\mb L, \mb S)$,
and $ \nabla_{\mb L} l(\mb L^k, \mb S^k) =  \nabla_{\mb S} l(\mb L^k, \mb
S^k) = \PO [\mb L^k + \mb S^k - \mb M]$. Since $\mb L$ and $\mb S$ are
decoupled in \eqref{eqn:ista}, equivalently we have
\begin{eqnarray}
  \mb L^{k+1} &=& \arg  \min_{\mb L} \norm{\mb L - \left( \mb L^k - \frac{1}{2} \PO[\mb L^k + \mb S^k - \mb M] \right)}{F}^2 + \lambda_L \norm{\mb L}{*}, \label{eqn:ista_nuc}\\
  \mb S^{k+1} &=& \arg  \min_{\mb S} \norm{\mb S - \left( \mb S^k - \frac{1}{2} \PO[\mb L^k + \mb S^k - \mb M] \right)}{F}^2 + \lambda_S \norm{\mb S}{1}. \label{eqn:ista_l1}
\end{eqnarray}
The solution to problem \eqref{eqn:ista_l1} can be given explicitly in terms of the proximal mapping of
$\norm{\cdot}{1}$ as introduced in Section 2.2,  i.e.,
$$
\mb S^{k+1} = \mc T_{\lambda_S/2}\left[\mb S^k - \frac{1}{2} \PO[\mb L^k + \mb S^k - \mb M]  \right].
$$
For a matrix $\mb X$ and any $\tau \ge 0$, let $\mc D_{\tau}(\mb X)$ denote the singular value thresholding operator $\mc D_{\tau} (\mb X) = \mb U \mc
T_{\tau}(\mb{\Sigma})\mb V^{\t}$, where $\mb X = \mb U \mb \Sigma \mb V^{\t}$ is
the singular value decomposition of $\mb X$. It is not difficult to show \cite{cai2010singular,ma2011fixed} that the solution to problem \eqref{eqn:ista_nuc} can be given explicitly by
$$
\mb L^{k+1} = \mc D_{\lambda_L/2}\left[\mb L^k - \frac{1}{2} \PO[\mb L^k + \mb S^k - \mb M]  \right].
$$
Algorithm \ref{alg:ista} summarizes our ISTA implementation for problem \eqref{eqn:SPCP}.

\begin{algorithm}[H]
   \caption{ISTA for problem \eqref{eqn:SPCP}}
   \label{alg:ista}
\begin{algorithmic}[1]
   \STATE {\bfseries Initialization:} $\mb L^0 = \mb 0$, $\mb S^0 = \mb 0$;
   \FOR{$k=0,\; 1,\; 2, \; \cdots$}
   \STATE $\mb L^{k+1} = \mc D_{\lambda_L/2}\left[ \mb L^k - \frac{1}{2} \PO [\mb L^k + \mb S^k - \mb M]   \right]$;
   \STATE $\mb S^{k+1} = \mc T_{\lambda_S/2}\left[ \mb S^k - \frac{1}{2} \PO [\mb L^k + \mb S^k - \mb M]   \right]$;
   \ENDFOR
\end{algorithmic}
\end{algorithm}
Regarding ISTA's speed of convergence, it can be proved that $f(\mb L^k, \mb
S^k) - f^\star = O(1/k)$, where $f^\star$ denotes the optimal value of
problem \eqref{eqn:SPCP}.

{\em Fast iterative soft-thresholding algorithm} (FISTA)
introduced in \cite{beck2009fast}, is an accelerated version
of ISTA, which incorporate a momentum step borrowed from Nesterov's optimal gradient scheme
\cite{nesterov1983method}. For FISTA, a better convergence result, $f(\mb L^k, \mb S^k) - f^\star = O(1/k^2)$, can be
achieved with a cost per iteration that is comparable to ISTA. Algorithm
\ref{alg:fista} summarizes our FISTA
implementation for problem \eqref{eqn:SPCP}.

\begin{algorithm}[H]
   \caption{FISTA for problem \eqref{eqn:SPCP}}
   \label{alg:fista}
\begin{algorithmic}[1]
   \STATE {\bfseries Initialization:} $\hat {\mb L}^0  = \mb L^0 = \mb 0$, $\hat {\mb S}^0  = \mb S^0 = \mb 0$, $t_0 =1$;
   \FOR{$k=0,\; 1,\; 2, \; \cdots$}
   \STATE $\mb L^{k+1} = \mc D_{\lambda_L/2}\left[ \hat{\mb L}^k - \frac{1}{2} \PO [\hat{\mb L}^k + \hat{\mb S}^k - \mb M]   \right]$;
   \STATE $\mb S^{k+1} = \mc T_{\lambda_S/2}\left[ \hat{\mb S}^k - \frac{1}{2} \PO [\hat{\mb L}^k + \hat{\mb S}^k - \mb M]   \right]$;
   \STATE $t^{k+1} = \frac{1+\sqrt{1+4(t^k)^2}}{2}$;
   \STATE $\hat {\mb L}^{k+1} = \mb L^{k+1} + \frac{t^k-1}{t^{k+1}}(\mb L^{k+1} - \mb L^k)$;
   \STATE $\hat {\mb S}^{k+1} = \mb S^{k+1} + \frac{t^k-1}{t^{k+1}}(\mb S^{k+1} - \mb S^k)$;
   \ENDFOR
\end{algorithmic}
\end{algorithm}
\paragraph{Partial SVD} In each iteration of either ISTA or FISTA, we
only need those singular values that are larger than $\lambda_S/2$ and their
corresponding singular vectors. Therefore, a partial SVD can be utilized to reduce
the computational burden of a full SVD. Since most partial SVD
software packages (e.g. PROPACK \cite{larsen2004propack}) require specifying in advance the number of top singular values and singular vectors to compute, we heuristically
determine this number (denoted as $sv^k$ at iteration $k$). Specifically,
let $d = \min \{m,n\}$, and $svp^k$ denote the number of computed singular
values that were larger than $\lambda_S/2$ in the $k$-th iteration. Similar to \cite{tao2011recovering}, in our
implementation, we start with $sv^0 = d/10$, and adjust $sv^k$ dynamically as
follows:
$$
sv^{k+1} = \begin{cases} \min\{svp^k+1, d\} &\mbox{if } svp^k < sv^k \\
\min\{svp^k+\mbox{round}(0.05d),d \} &\mbox{otherwise}. \end{cases}
$$

\subsection{Foreground-background separation in surveillance video}

In surveillance videos, due to the strong correlation between frames, it is natural to model the background as low rank; while  foreground objects, such as cars or pedestrians, that normally occupy only a fraction of the video, can be treated as sparse. So, if we stack each frame as a column in the data matrix $\mb M_0$, it is reasonable to assume that $\mb M_0 \approx \mb L_0 + \mb S_0$, where $\mb L_0$ captures the background and $\mb S_0$ represents the foreground movements. Here, we solved problem \eqref{eqn:SPCP} for videos introduced in \cite{li2004statistical} and \cite{jacobs2007consistent}. The observed entries were sampled uniformly with ratio $\rho$ chosen respectively as $1$, $0.8$ and $0.6$.

Table \ref{tab:surv} summarizes the numerical performances of FW-T, ISTA and FISTA in terms of the iteration number and running time (in seconds).
As can be observed, our FW-T method is more efficient than ISTA and FISTA, and the advantage becomes more prominent as the size of the data grows and the observations are more compressed (with smaller sampling ratio $\rho$).
Even though the FW-T method took more iterations than FISTA and in many cases than ISTA, it took less time in many cases but one due to its low per-iteration cost. To illustrate this more clearly, in Figure \ref{fig:time_per_iter}, we plot the per-iteration cost of these three methods on the Airport and Square videos as a function of the number of frames. The computational cost of FW-T scales linearly with the size of the data, whereas the cost of the other methods increases superlinearly.
Another observation is that as the number of measurements decreases, the iteration numbers of both ISTA and FISTA methods grow substantially, while those of the FW-T method remain quite stable. This explains the more favorable behavior of the FW-T method when $\rho$ is small.
In Figure \ref{fig:video}, frames of the original videos, the backgrounds and the foregrounds produced by the FW-T method are presented, and the separation achieved is quite satisfactory.

\begin{figure}[tb]
\centerline{
\begin{minipage}{3in}
\centerline{\includegraphics[width=2.5in]{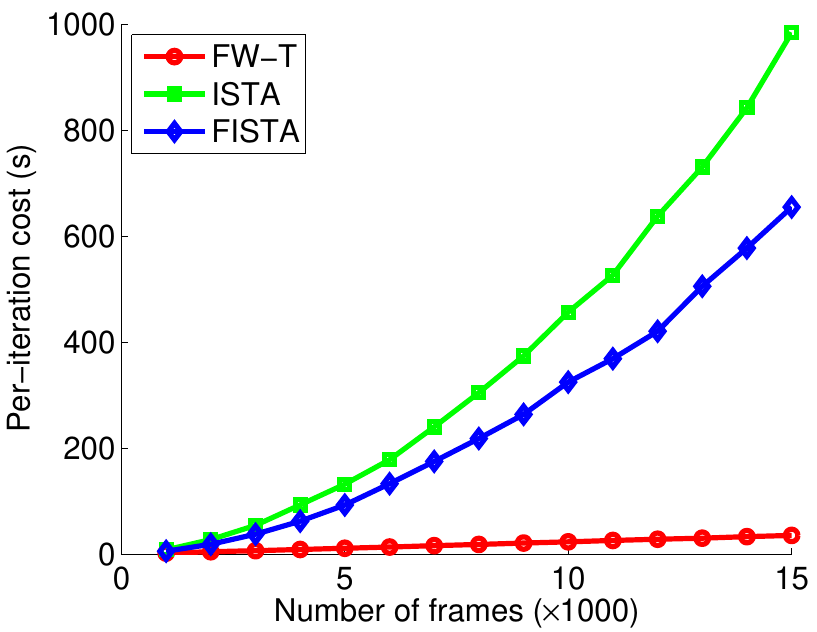}}
\end{minipage}
\hspace{-5mm}
\begin{minipage}{3in}
\centerline{\includegraphics[width=2.5in]{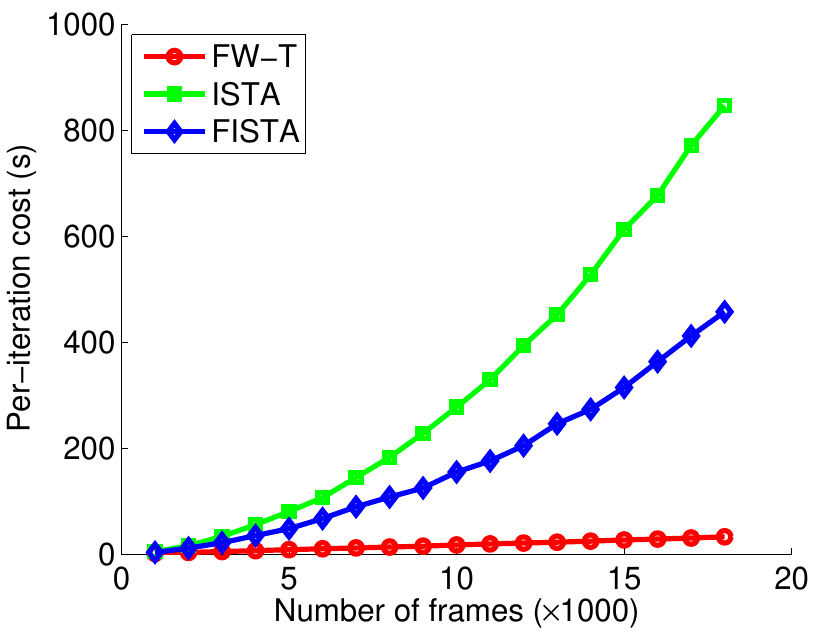}}
\end{minipage}
}
\vspace{3mm}
\centerline{
\begin{minipage}{3in}
\centerline{\textbf{Airport}}
\end{minipage}
\hspace{-5mm}
\begin{minipage}{3in}
\centerline{\textbf{Square}}
\end{minipage}
}
%\centerline{\includegraphics[width=2in]{fig/timeIter.pdf} }
\caption{\textbf{Per-iteration cost vs. the number of frames in Airport and Square videos with full observation.} The per-iteration cost of our FW-T method grows linearly with the size of data, in contrast with the superlinear per-iteration cost of ISTA and FISTA. That makes the FW-T method more advantageous or may even be the only feasible choice for large problems.}
\label{fig:time_per_iter}
\end{figure}

\begin{figure}[htb!]
\centerline{
\begin{minipage}{2.7in}
\centerline{\includegraphics[width=2.4in,height=0.6in]{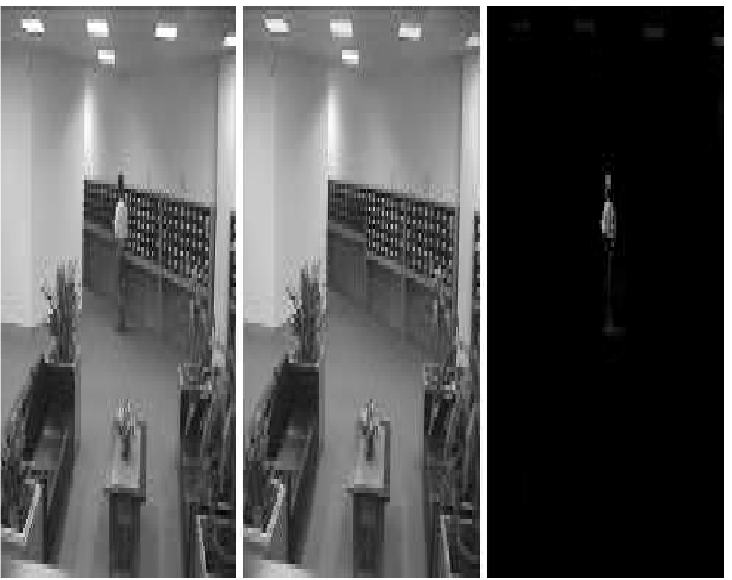}}
\end{minipage}
\begin{minipage}{2.7in}
\centerline{\includegraphics[width=2.4in,height=0.6in]{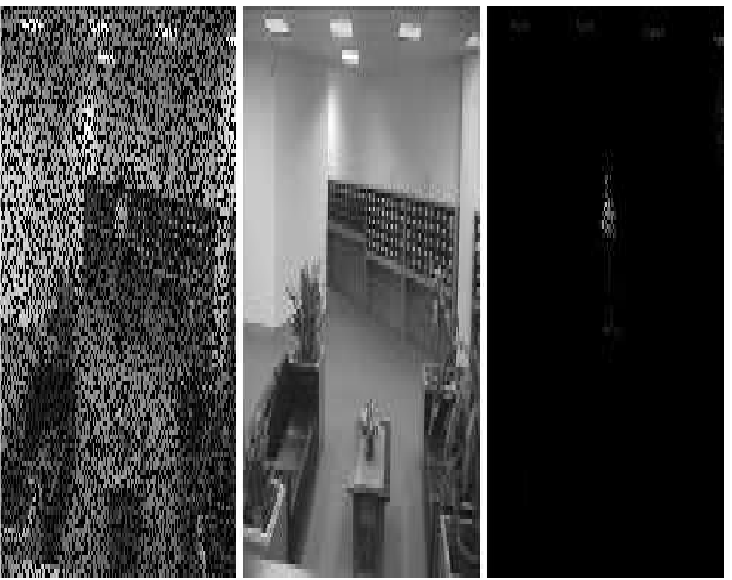}}
\end{minipage}
}
\vspace{1.5mm}
\centerline{
\begin{minipage}{2.7in}
\centerline{\includegraphics[width=2.4in,height=0.6in]{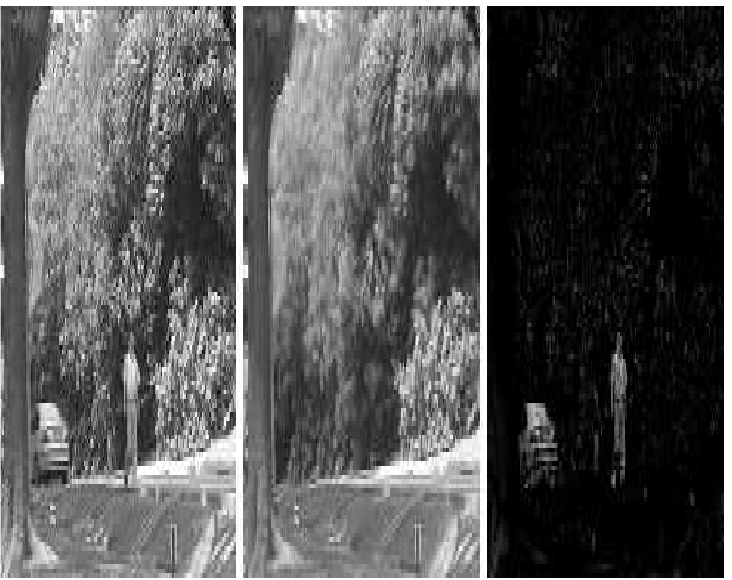}}
\end{minipage}
\begin{minipage}{2.7in}
\centerline{\includegraphics[width=2.4in,height=0.6in]{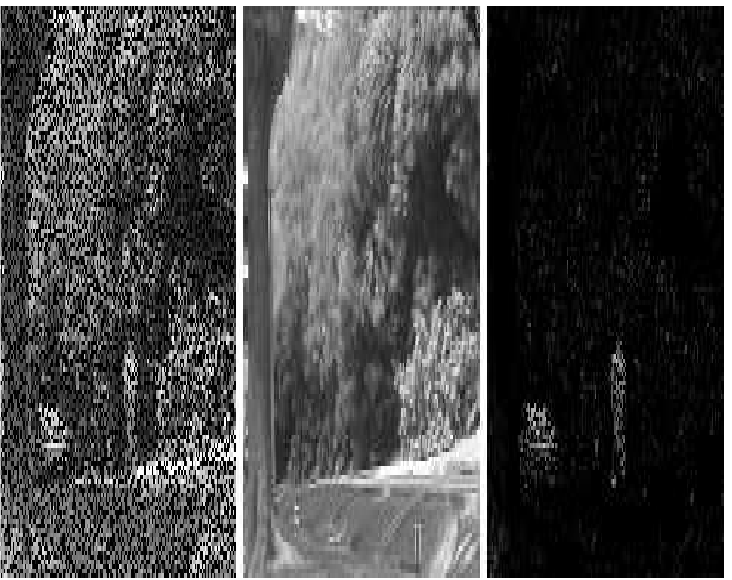}}
\end{minipage}
}
\vspace{1.5mm}
\centerline{
\begin{minipage}{2.7in}
\centerline{\includegraphics[width=2.4in,height=0.6in]{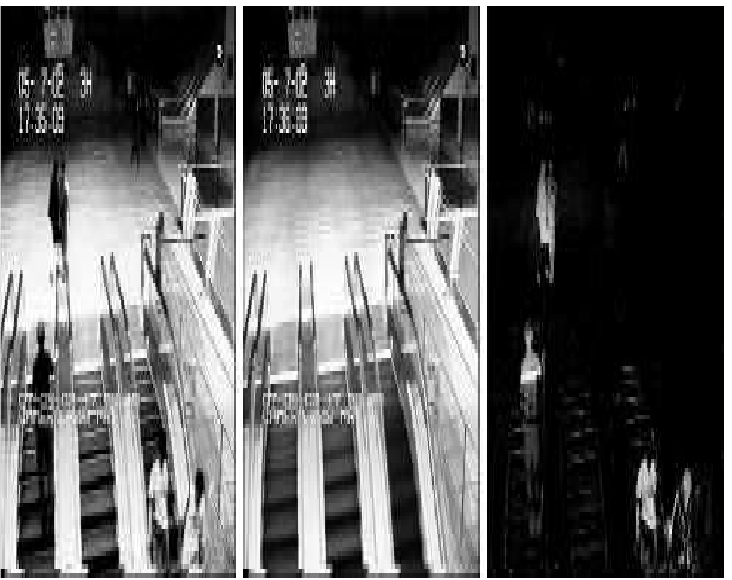}}
\end{minipage}
\begin{minipage}{2.7in}
\centerline{\includegraphics[width=2.4in,height=0.6in]{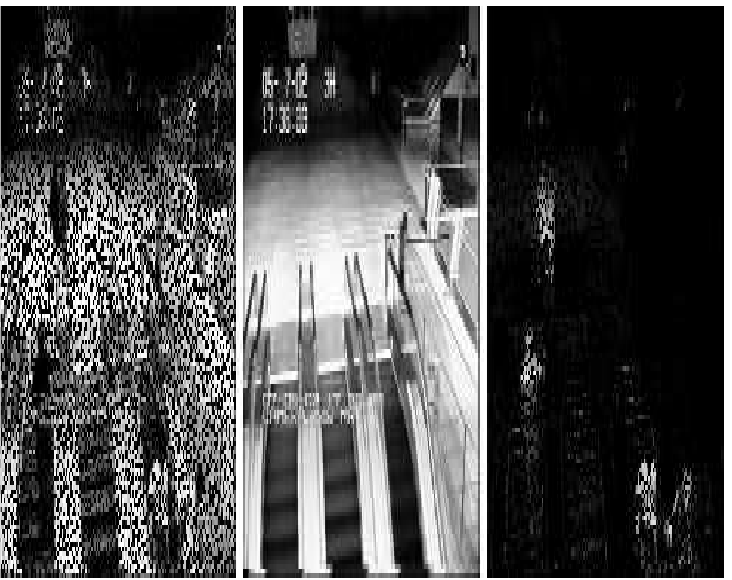}}
\end{minipage}
}
\vspace{1.5mm}
\centerline{
\begin{minipage}{2.7in}
\centerline{\includegraphics[width=2.4in,height=0.6in]{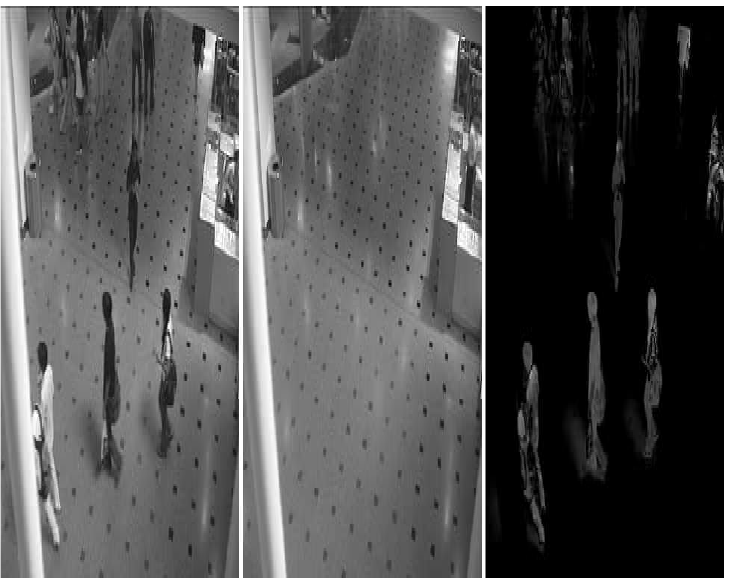}}
\end{minipage}
\begin{minipage}{2.7in}
\centerline{\includegraphics[width=2.4in,height=0.6in]{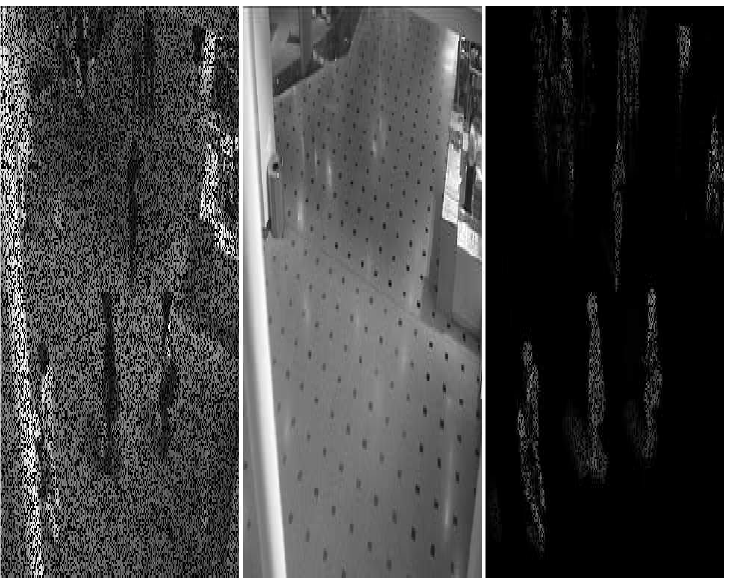}}
\end{minipage}
}
\vspace{1.5mm}
\centerline{
\begin{minipage}{2.7in}
\centerline{\includegraphics[width=2.4in,height=0.6in]{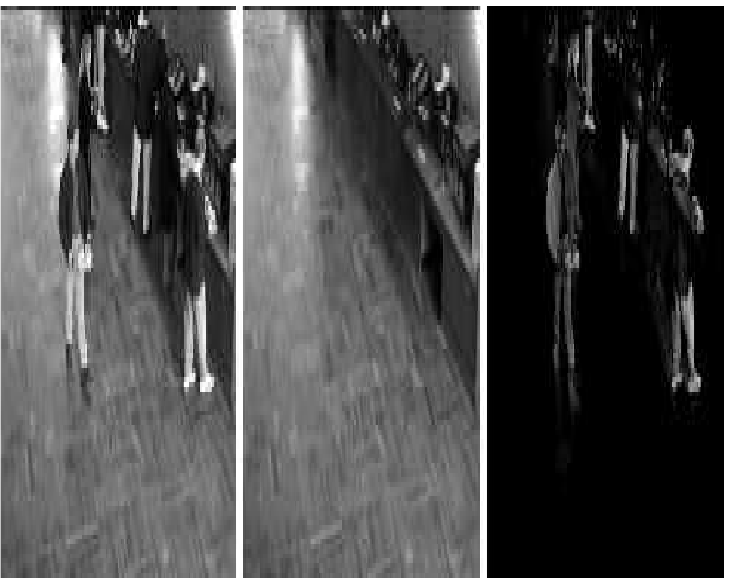}}
\end{minipage}
\begin{minipage}{2.7in}
\centerline{\includegraphics[width=2.4in,height=0.6in]{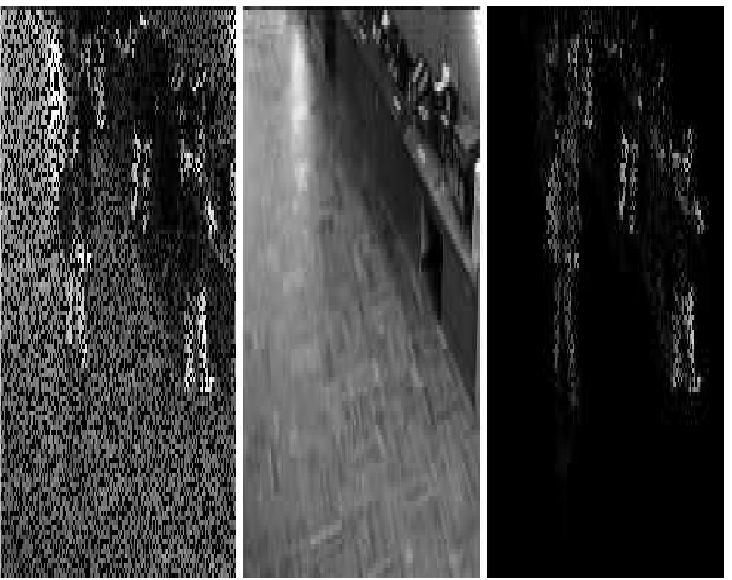}}
\end{minipage}
}
\vspace{1.5mm}
\centerline{
\begin{minipage}{2.7in}
\centerline{\includegraphics[width=2.4in,height=0.6in]{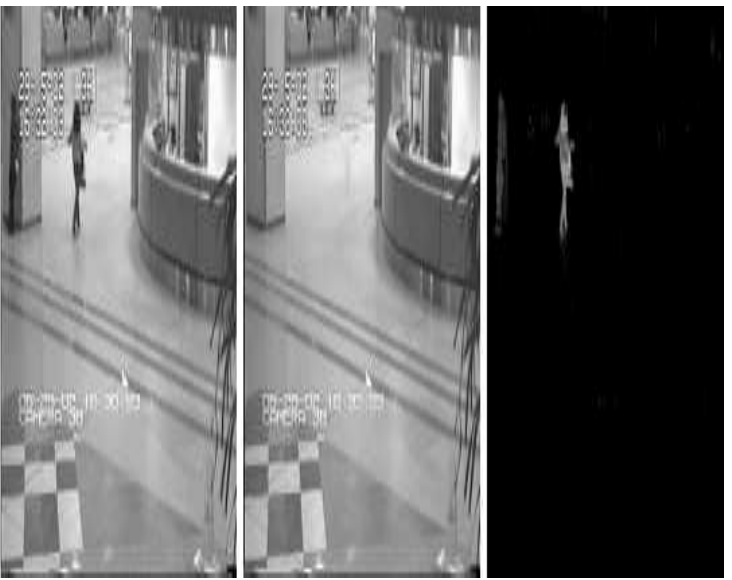}}
\end{minipage}
\begin{minipage}{2.7in}
\centerline{\includegraphics[width=2.4in,height=0.6in]{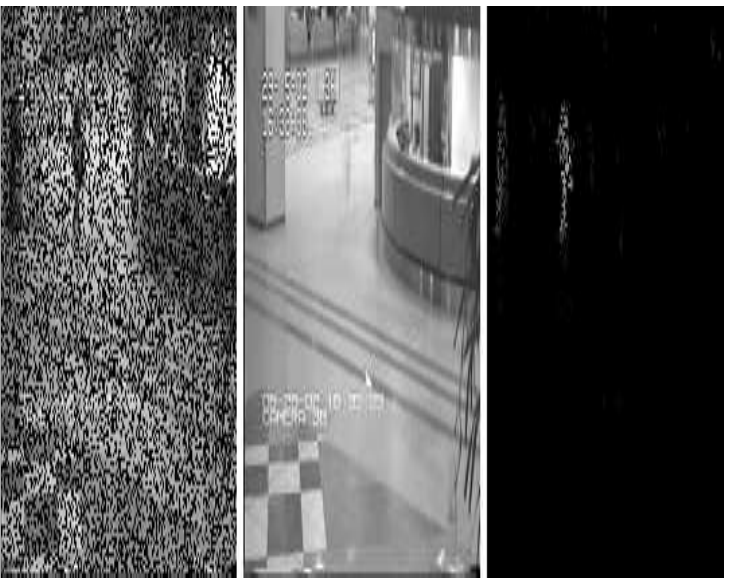}}
\end{minipage}
}
\vspace{1.5mm}
\centerline{
\begin{minipage}{2.7in}
\centerline{\includegraphics[width=2.4in,height=0.6in]{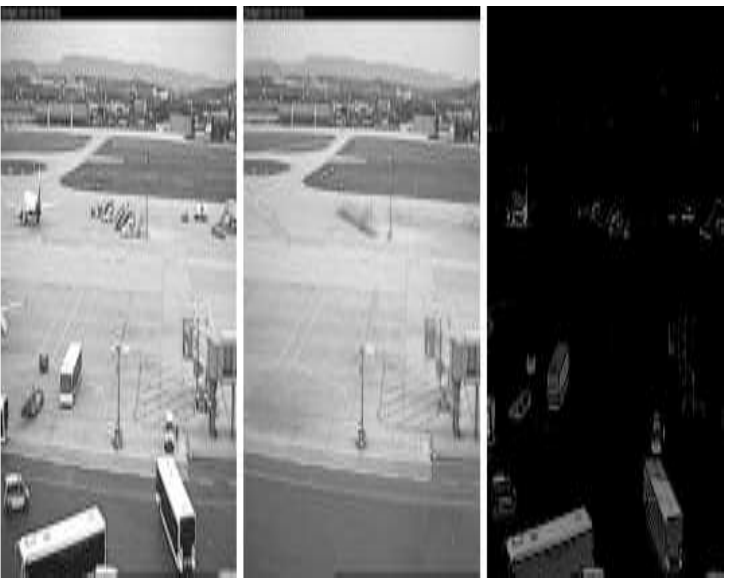}}
\end{minipage}
\begin{minipage}{2.7in}
\centerline{\includegraphics[width=2.4in,height=0.6in]{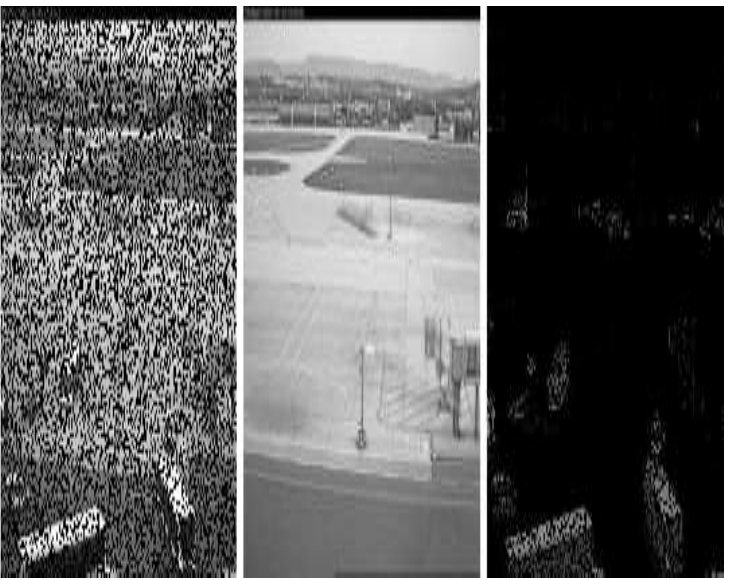}}
\end{minipage}
}
\vspace{1.5mm}
\centerline{
\begin{minipage}{2.7in}
\centerline{\includegraphics[width=2.4in,height=0.6in]{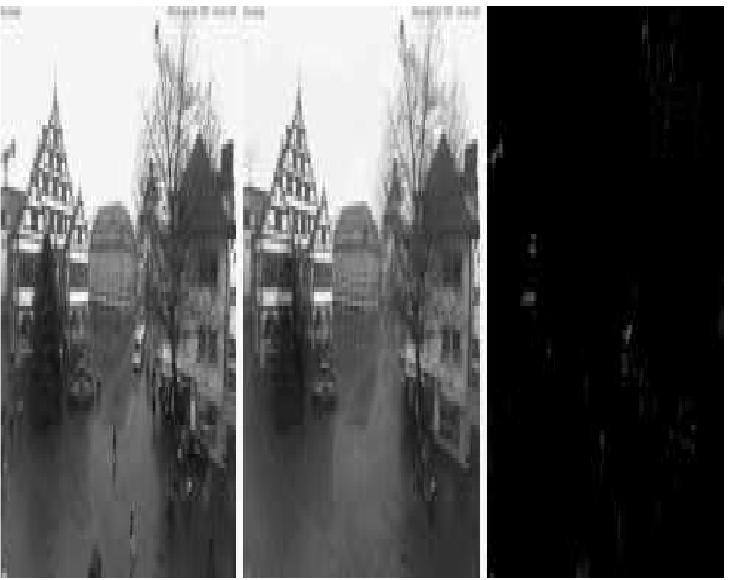}}
\end{minipage}
\begin{minipage}{2.7in}
\centerline{\includegraphics[width=2.4in,height=0.6in]{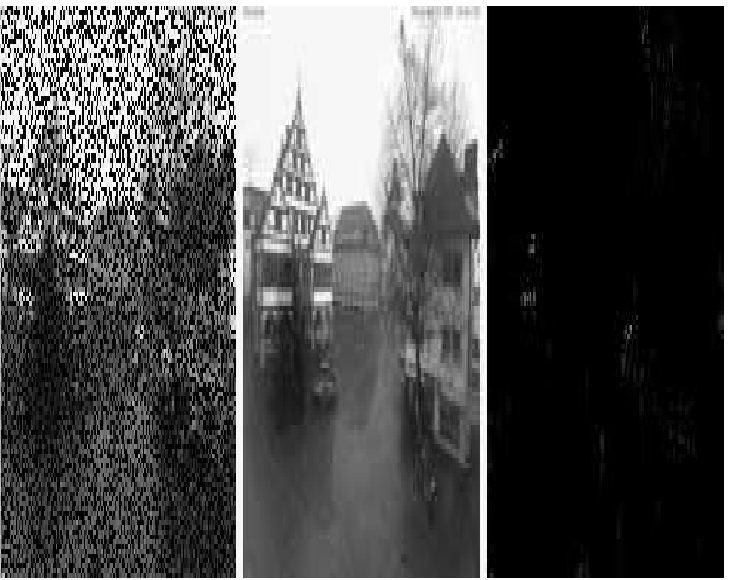}}
\end{minipage}
}
% words
\vspace{4mm}
\centerline{
\begin{minipage}{2.7in}
{\hspace{10mm} $\mb M_0$\hspace{17mm}$\hat{\mb L}$\hspace{18mm}$\hat{\mb S}$}
\end{minipage}
\begin{minipage}{2.7in}
{\hspace{8mm}$\PO[\mb M_0]$\hspace{13mm}$\hat{\mb L}$\hspace{17mm}$\hat{\mb S}$}
\end{minipage}
}
% description
\caption{\textbf{Surveillance videos.} The videos from top to bottom are respectively Lobby, Campus, Escalator, Mall, Restaurant, Hall, Airport and Square. The left panel presents videos with full observation ($\rho=1$) and the right one presents videos with partial observation ($\rho = 0.6$). Visually, the low-rank component successfully recovers the background and the sparse one captures the moving objects (e.g. vehicles, pedestrians) in the foreground.}
\label{fig:video}
\end{figure}

\begin{table}
\centering
\caption{\textbf{Comparisons of FW-T, ISTA and FISTA on surveillance video data.} The advantage of our FW-T method becomes prominent when the data are at large scale and compressed (i.e. the small $\rho$ scenarios).}
{\tabulinesep = .07in
\begin{tabu}{l |[1pt] c |[1pt] c c c c c c}
\hline
&& \multicolumn{2}{c}{FW-T} & \multicolumn{2}{c}{ISTA} & \multicolumn{2}{c}{FISTA} \\
Data & $\rho$ &  iter. & time  & iter. & time  & iter. & time  \\
\hline
\hline
Lobby                                & $1.0$ & 96  & {1.94e+02} & 144 & 3.64e+02  & 41  & \textbf{1.60e+02} \\
{\footnotesize ($20480\times 1000$)} & $0.8$ & 104 & \textbf{2.33e+02} & 216 & 1.03e+03 & 52  & 3.55e+02 \\
                                     & $0.6$ & 133 & \textbf{3.12e+02} & 380 & 1.67e+03 & 74  & 5.10e+02 \\
%                          & $0.4$ & 158 & \textbf{368} & 812 & 4310 & 113 & 791 \\
%                          & $0.2$ & 211 & \textbf{455} & --  & --   & 208 & 3490 \\
\hline
\hline
Campus                               & $1.0$ & 45 &\textbf{1.56e+02} & 78&1.49e+03 &23 & 4.63e+02 \\
{\footnotesize $(20480\times 1439)$} & $0.8$ & 44 &\textbf{1.57e+02} & 122& 2.34e+03 &30 & 6.45e+02 \\
                                     & $0.6$ & 41 &\textbf{1.39e+02} & 218&4.27e+03 &43 &1.08e+03 \\
                          % & $0.4$ & 0 &0 & 0 &0 &0 & 0 \\
                          % & $0.2$ & 0 &0 & 0 &0 &0 & 0 \\
\hline
\hline
Escalator                          & $1.0$ & 81 &\textbf{7.40e+02} & 58  & 4.19e+03 &25 & 2.18e+03 \\
{\footnotesize $(20800\times 3417)$} & $0.8$ & 80 &\textbf{7.35e+02} & 90  & 8.18e+03 &32 & 3.46e+03 \\
                                   & $0.6$ & 82 &\textbf{7.68e+02} & 162 & 1.83e+04 &43 & 5.73e+03 \\
                          %& $0.4$ & 0 &0 & 0 &0 &0 & 0 \\
                          %& $0.2$ & 0 &0 & 0 &0 &0 & 0 \\
\hline
\hline
Mall                                  & $1.0$ & 38 &\textbf{4.70e+02} & 110 &5.03e+03 &35 & 1.73e+03 \\
{\footnotesize $(81920 \times 1286)$} & $0.8$ & 35 &\textbf{4.58e+02} & 171 &7.32e+03 &44 & 2.34e+03 \\
                                      & $0.6$ & 44 &\textbf{5.09e+02} & 308 &1.31e+04 &62 & 3.42e+03 \\
                          %& $0.4$ & 0 &0 & 0 &0 &0 & 0 \\
                          %& $0.2$ & 0 &0 & 0 &0 &0 & 0 \\
\hline
\hline
Restaurant                            & $1.0$ & 70 & \textbf{5.44e+02} & 52  & 3.01e+03 & 20 & 1.63e+03 \\
{\footnotesize $(19200\times 3055)$}  & $0.8$ & 74 & \textbf{5.51e+02} & 81  & 4.84e+03 & 26 & 1.82e+03 \\
                                      & $0.6$ & 76 & \textbf{5.73e+02} & 144 & 9.93e+03 & 38 & 3.31e+03 \\
                                   % & $0.4$ & 0 &0 & 0 &0 &0 & 0 \\
                                   % & $0.2$ & 0 &0 & 0 &0 &0 & 0 \\
\hline
\hline
Hall                                 & $1.0$ & 60 &\textbf{6.33e+02} & 52  & 2.98e+03 & 21 & 1.39e+03 \\
{\footnotesize $(25344\times 3584)$} & $0.8$ & 62 &\textbf{6.52e+02} & 81  & 6.45e+03 & 28 & 2.90e+03 \\
                                     & $0.6$ & 70 &\textbf{7.43e+02} & 144 & 1.42e+04 & 39 & 4.94e+03 \\
                          %& $0.4$ & 0 &0 & 0 &0 &0 & 0 \\
                          %& $0.2$ & 0 &0 & 0 &0 &0 & 0 \\
\hline
\hline
Airport                                & $1.0$ & 130 & \textbf{6.42e+03} & 29 & 2.37e+04 & 14 & 1.37e+04  \\
{\footnotesize $(25344 \times 15730)$} & $0.8$ & 136 & \textbf{6.65e+03} & 45 & 6.92e+04 & 18 & 4.27e+04 \\
                                       & $0.6$ & 154 & \textbf{7.72e+03} & 77 & 1.78e+05 & 24 & 7.32e+04 \\
                          % & $0.4$ & 0 &0 & 0 &0 &0 & 0 \\
                          % & $0.2$ & 0 &0 & 0 &0 &0 & 0 \\
\hline
\hline
Square                                & $1.0$ & 179 & \textbf{1.24e+04} & 29 & 3.15e+04 & 13 & 1.51e+04 \\
{\footnotesize $(19200\times 28181)$} & $0.8$ & 181 & \textbf{1.26e+04} & 44 & 1.04e+05 & 17 & 6.03e+04 \\
                                      & $0.6$ & 191 & \textbf{1.31e+04} & 78 & 2.63e+05 & 22 & 9.88e+05 \\
                          %& $0.8$ & 0 &0 & 0 &0 &0 & 0 \\
                          %& $1.0$ & 0 &0 & 0 &0 &0 & 0 \\
\hline
\end{tabu}}
\label{tab:surv}
\end{table}

\begin{table}
\centering
\caption{\textbf{Comparisons of FW-T, ISTA and FISTA on YaleB face data.} The number of frames, 65, is relatively small for this application. This disables the FW-T method to significantly benefit from its linear per-iteration cost and consequently the FISTA method consistently has a better performance.
%Consequently, the FW-T method becomes better only when the sampling ratio is quite low.
}
{\tabulinesep = .055in
\begin{tabu}{ l |[1pt] c |[1pt] c c c c c c}
\hline
&& \multicolumn{2}{c}{FW-T} & \multicolumn{2}{c}{ISTA} & \multicolumn{2}{c}{FISTA} \\
Data & $\rho$ &  iter. & time  & iter. & time  & iter. & time  \\
\hline
\hline
YaleB01 & $1.0$ & 65 & 34.0 & 49  & 21.4 & 17  & \textbf{8.69} \\
{\footnotesize $(32256 \times 65)$}        & $0.9$ & 68 & 35.6 & 59  & 23.9 & 19  & \textbf{8.62} \\
        & $0.8$ & 79 & 42.2 & 76  & 35.3 & 22  & \textbf{10.9} \\
        & $0.7$ & 76 & 39.9 & 97  & 44.0 & 25  & \textbf{11.1} \\
        & $0.6$ & 71 & 37.5 & 127 & 50.2 & 29  & \textbf{12.9} \\
        & $0.5$ & 80 & 40.5 & 182 & 77.9 & 35  & \textbf{15.2} \\
        %& $0.4$ & 70 & 35.7 & 259 & 107  & 42  & \textbf{18.3} \\
        %& $0.3$ & 74 & 36.3 & 428 & 178  & 55  & \textbf{23.0} \\
        %& $0.2$ & 63 & 31.4 & 763 & 307  & 74  & \textbf{29.7} \\
        %& $0.1$ & 32 & \textbf{16.1} & --  & --   & 102 & 43.5 \\
\hline
\hline
YaleB02 & $1.0$ & 64 & 34.6 & 51  & 19.2  & 18  & \textbf{7.31} \\
{\footnotesize $(32256 \times 65)$}        & $0.9$ & 64 & 26.8 & 61  & 22.6  & 20  & \textbf{7.32} \\
        & $0.8$ & 71 & 33.9 & 78  & 27.7  & 22  & \textbf{8.61} \\
        & $0.7$ & 71 & 31.3 & 99  & 36.6  & 26  & \textbf{11.0} \\
        & $0.6$ & 73 & 36.6 & 132 & 53.7  & 30  & \textbf{12.4} \\
        & $0.5$ & 63 & 28.0 & 177 & 64.6  & 35  & \textbf{13.4} \\
%        & $0.4$ & 66 & 29.2 & 267 & 92.9  & 45  & \textbf{16.3} \\
%        & $0.3$ & 70 & 27.2 & 445 & 155   & 56  & \textbf{19.1} \\
%        & $0.2$ & 62 & \textbf{25.6} & 837 & 270   & 78  & 27.5 \\
%        & $0.1$ & 22 & \textbf{9.37} & --  & --   & 106  & 38.1 \\
\hline
\hline
YaleB03 & $1.0$ & 62 & 26.0 & 49  & 16.6 & 18  & \textbf{6.00} \\
{\footnotesize $(32256 \times 65)$}        & $0.9$ & 71 & 27.5 & 62  & 20.3 & 20  & \textbf{6.43} \\
        & $0.8$ & 69 & 30.0 & 78  & 26.0 & 22  & \textbf{8.32} \\
        & $0.7$ & 78 & 31.5 & 101 & 32.9 & 26  & \textbf{9.00} \\
        & $0.6$ & 73 & 28.7 & 132 & 40.4 & 30  & \textbf{10.6} \\
        & $0.5$ & 70 & 28.0 & 181 & 60.3 & 36  & \textbf{12.8} \\
%        & $0.4$ & 68 & 34.2 & 266 & 94.8 & 44  & \textbf{16.8} \\
%        & $0.3$ & 62 & 26.8 & 424 & 144  & 55  & \textbf{20.6} \\
%        & $0.2$ & 60 & \textbf{25.3} & 789 & 290  & 75  & 28.1 \\
%        & $0.1$ & 31 & \textbf{11.8} & --  & --   & 101 & 38.6 \\
\hline
\hline
YaleB04 & $1.0$ & 63 & 28.5 & 47  & 16.6 & 17  & \textbf{6.35} \\
{\footnotesize $(32256 \times 65)$} & $0.9$ & 67 & 28.7 & 58  & 23.1 & 19  & \textbf{7.98} \\
        & $0.8$ & 68 & 31.7 & 72  & 26.3 & 23  & \textbf{9.39} \\
        & $0.7$ & 69 & 30.7 & 92  & 35.9 & 26  & \textbf{9.84} \\
        & $0.6$ & 71 & 29.4 & 124 & 40.0 & 29  & \textbf{10.1} \\
        & $0.5$ & 74 & 29.4 & 174 & 67.3 & 36  & \textbf{14.3} \\
 %       & $0.4$ & 71 & 29.0 & 255 & 88.5 & 42  & \textbf{14.9} \\
%        & $0.3$ & 71 & 30.2 & 417 & 157  & 54  & \textbf{22.0} \\
%        & $0.2$ & 67 & 31.0 & 760 & 272  & 74  & \textbf{29.6} \\
%        & $0.1$ & 40 & \textbf{15.7} & --  & --   & 99  & 43.7 \\
\hline
\end{tabu}}
\label{tab:face}
\end{table}

\subsection{Shadow and specularity removal from face images}
Images taken under varying illumination can also be modeled as the superposition of low-rank and sparse components. Here, the data matrix $\mb M_0$ is again formed by stacking each image as a column. The low-rank term $\mb L_0$ captures the smooth variations \cite{basri2003lambertian}, while  the sparse term $\mb S_0$ represents cast shadows and specularities \cite{wright2009robust,Zhang2013-ICCV}. CPCP can be used to remove the shadows and specularities \cite{candes2011robust,Zhang2013-ICCV}. Here, we solved problem \eqref{eqn:penalized_prob} for YaleB face images \cite{georghiades2001few}. %, and images rendered from 3D
%triangulated face models from \cite{Bosphorus}.
Table \ref{tab:face} summarizes the numerical performances of FW-T, ISTA and FISTA.
%where ``--'' indicates that the number of iterations exceeded the cut-off of 1000.
%As the number of frames is relatively small for this application, the FW-T method is better only when the sampling ratio is quite low.
Similar to the observation made regarding the above surveillance video experiment, the number of iterations required by ISTA and FISTA grows much faster than it does for the FW-T method when $\rho$ decreases. However, unlike in those tests, where the number of frames in each dataset was at least several thousand, the number of frames here is just $65$. This prevents the FW-T method from significantly benefiting from its linear per-iteration cost and consequently, while FW-T still outperforms ISTA for values of $\rho \le 0.7$, the FISTA method is always the fastest.
In Figure \ref{fig:face}, the original images, the low-rank and the sparse parts produced by the FW-T method are presented. Visually, the recovered low-rank component is smoother and better conditioned for face recognition than the original image, while the sparse component corresponds to shadows and specularities.
\begin{figure}[htb!]
\centerline{
\begin{minipage}{2.7in}
\centerline{\includegraphics[width=2.4in,height=0.8in]{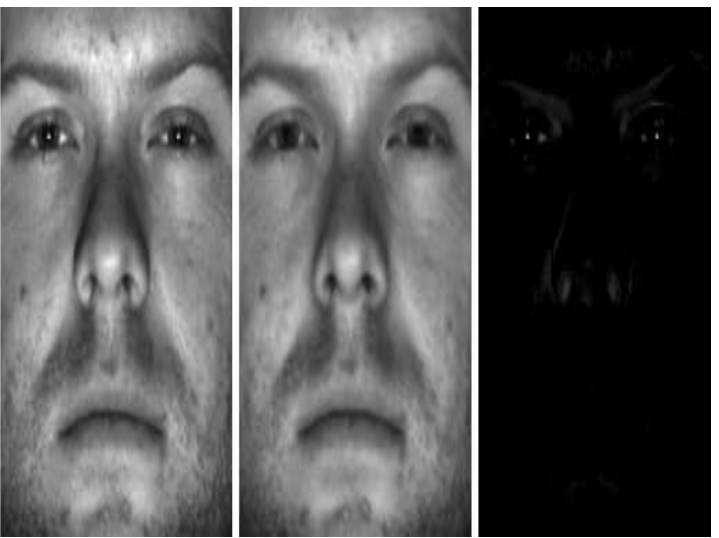}}
\end{minipage}
\begin{minipage}{2.7in}
\centerline{\includegraphics[width=2.4in,height=0.8in]{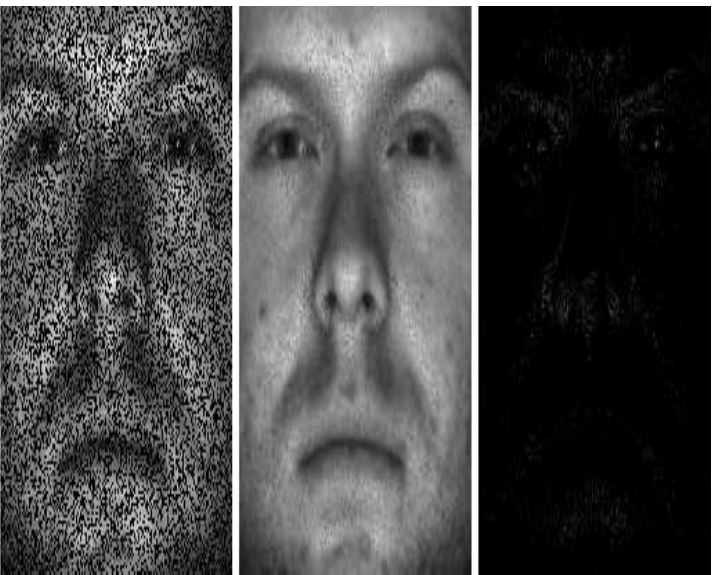}}
\end{minipage}
}
\vspace{2mm}
\centerline{
\begin{minipage}{2.7in}
\centerline{\includegraphics[width=2.4in,height=0.8in]{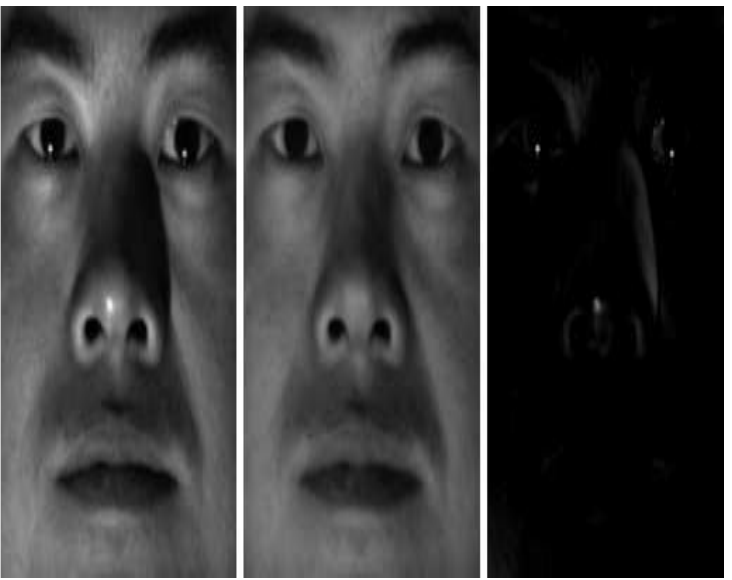}}
\end{minipage}
\begin{minipage}{2.7in}
\centerline{\includegraphics[width=2.4in,height=0.8in]{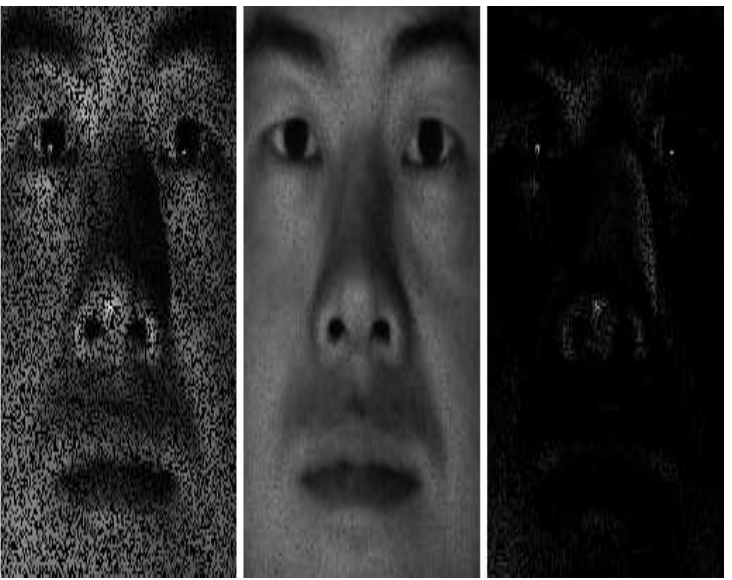}}
\end{minipage}
}
\vspace{2mm}
\centerline{
\begin{minipage}{2.7in}
\centerline{\includegraphics[width=2.4in,height=0.8in]{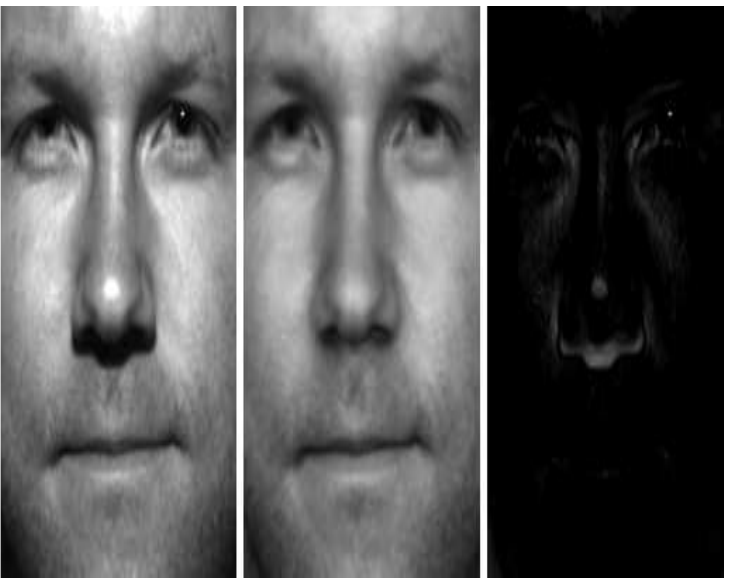}}
\end{minipage}
\begin{minipage}{2.7in}
\centerline{\includegraphics[width=2.4in,height=0.8in]{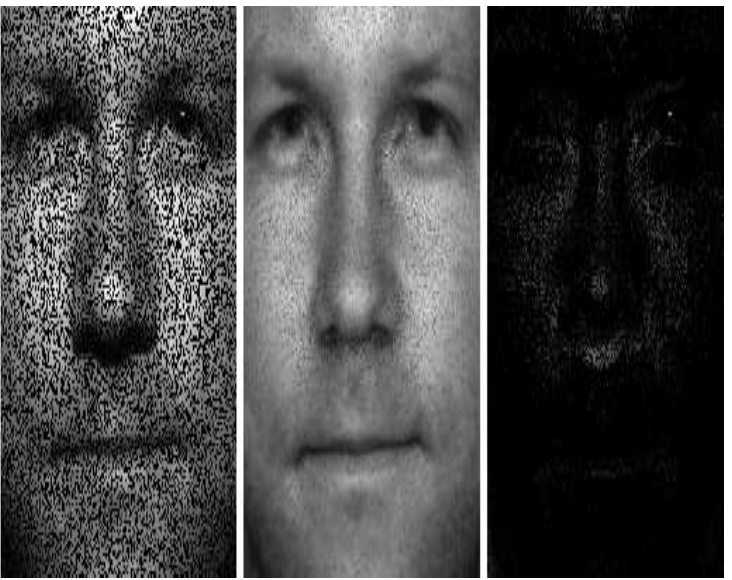}}
\end{minipage}
}
\vspace{2mm}
\centerline{
\begin{minipage}{2.7in}
\centerline{\includegraphics[width=2.4in,height=0.8in]{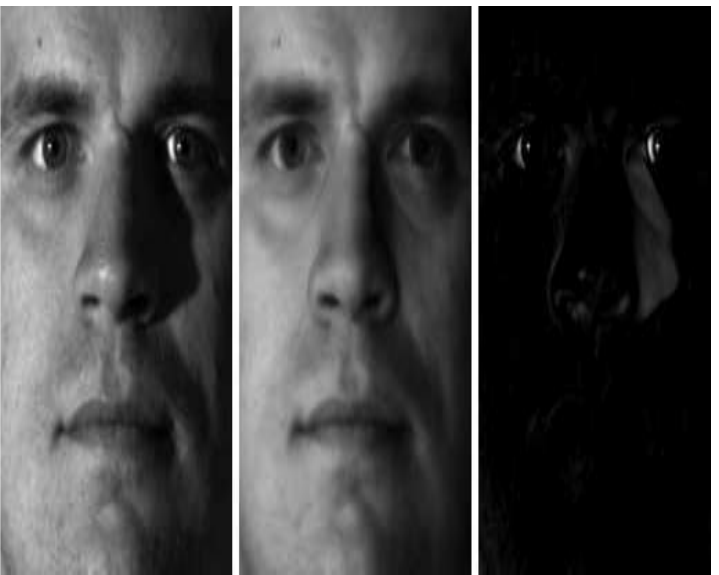}}
\end{minipage}
\begin{minipage}{2.7in}
\centerline{\includegraphics[width=2.4in,height=0.8in]{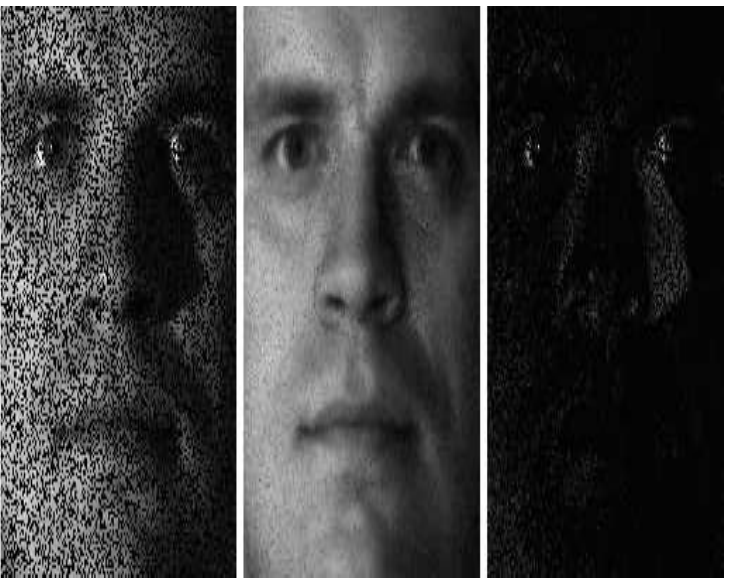}}
\end{minipage}
}
\vspace{4mm}
\centerline{
\begin{minipage}{2.7in}
{\hspace{10mm} $\mb M_0$\hspace{17mm}$\hat{\mb L}$\hspace{18mm}$\hat{\mb S}$}
\end{minipage}
\begin{minipage}{2.7in}
{\hspace{8mm}$\PO[\mb M_0]$\hspace{13mm}$\hat{\mb L}$\hspace{17mm}$\hat{\mb S}$}
\end{minipage}
}
\caption{\textbf{Face images.} The pictures from top to bottom are respectively YaleB01, YaleB02, YaleB03 and  YaleB04 face images. The left panel presents the case with full observation ($\rho = 1$), while the right panel presents the case with partial observation ($\rho = 0.6$). Visually, the recovered low-rank component is smoother and better conditioned for face recognition than the original image, while the sparse component corresponds to shadows and specularities.}
\label{fig:face}
\end{figure}

\section{Discussion} \label{sec:conclusion}
In this paper, we have proposed scalable algorithms called Frank-Wolfe-Projection (FW-P) and Frank-Wolfe-Thresholding (FW-T) for norm constrained and penalized versions of CPCP. Essentially, these methods combine classical ideas in Frank-Wolfe and Proximal methods to achieve linear per-iteration cost, $O(1/k)$ convergence in function value and practical efficiency in updating the sparse component. Extensive numerical experiments were conducted on computer vision related applications of CPCP, which demonstrated the great potential of our methods for dealing with problems of very large scale. Moreover, the general idea of leveraging different methods to deal with different functions may be valuable for other demixing problems.

We are also aware that though our algorithms are extremely efficient in the beginning iterations and quickly arrive at an approximate solution of practical significance, they become less competitive in solutions of very high accuracy, due to the nature of Frank-Wolfe. This suggests further hybridization under our framework (e.g. using nonconvex approaches to handle the nuclear norm) might be utilized in certain applications (see \cite{laue2012hybrid} for research in that direction).

\section*{Acknowledgements} We are grateful to the associate editor Chen Greif and three anonymous reviewers for their helpful suggestions and comments that substantially improve the paper.

\bibliographystyle{ieeetr}
\bibliography{refs}

\end{document}

%% file: cm_defs_simax.tex
\usepackage{amsfonts}
\usepackage{amssymb,amsbsy}
\usepackage{amsmath}
\usepackage{amsfonts, amscd}
\usepackage{algorithm, algorithmic}
\usepackage{color}
\usepackage{url}
\usepackage{verbatim}
\usepackage{setspace}
\usepackage[abs]{overpic}
\usepackage{graphicx}
\usepackage{subfigure}
\usepackage{bbm, bm}
\usepackage{accents}
\usepackage{mleftright}
\usepackage{dsfont}
\newtheorem{example*}{Example}
\numberwithin{equation}{section}
\renewcommand{\theequation}{\thesection.\arabic{equation}}

\newtheorem{remark}{Remark}

\newcommand{\mb}{\boldsymbol}
\newcommand{\mc}{\mathcal}

\newcommand{\norm}[2]{\left\| #1 \right\|_{#2}}
\newcommand{\reals}{\mathbb{R}}
\newcommand{\<}{\langle}
\renewcommand{\>}{\rangle}
\newcommand{\innerprod}[2]{\left\< #1, #2 \right\>}
\newcommand{\set}[1]{\left\{ #1 \right\}}
\newcommand{\eps}{\varepsilon}

\renewcommand\t{{\ensuremath{\scriptscriptstyle{\top}}}}

\newcommand{\PQ}{\mc P_Q}
\newcommand{\PO}{\mathcal{P}_\Omega}
\newcommand{\Proj}{\mc P}
\renewcommand\t{{\ensuremath{\scriptscriptstyle{\top}}}} 